\documentclass[a4paper,11pt]{amsart}

\usepackage[matrix,arrow,tips,curve]{xy}
\usepackage[english]{babel}
\usepackage{amsmath}
\usepackage{amssymb}

\usepackage{mathrsfs}

\usepackage{enumerate}

\newtheorem{thm}{Theorem}[section]
\newtheorem{lemma}[thm]{Lemma}

\newtheorem{corollary}[thm]{Corollary}
\newtheorem{proposition}[thm]{Proposition}

\newtheorem*{thm*}{Theorem}

\theoremstyle{definition}
\newtheorem{remark}[thm]{Remark}
\newtheorem{notation}[thm]{Notation}
\newtheorem{example}[thm]{Example}
\newtheorem{parg}[thm]{}

\newcommand{\ph}{\varphi}
\newcommand{\w}{\widetilde}
\newcommand{\eps}{\varepsilon}
\newcommand{\ma}{\mathcal}
\newcommand{\la}{\longrightarrow}
\newcommand\map{\dasharrow}

\newcommand{\pr}{\mathbb{P}}
\newcommand{\Q}{\mathbb{Q}}
\newcommand{\R}{\mathbb{R}}
\newcommand{\Z}{\mathbb{Z}}
\newcommand{\N}{\mathcal{N}_1}
\newcommand{\Nu}{\mathcal{N}^1}
\newcommand{\Gr}{\operatorname{Gr}}

\newcommand{\Pic}{\operatorname{Pic}}

\newcommand{\NE}{\operatorname{NE}}
\newcommand{\Exc}{\operatorname{Exc}}

\newcommand{\Eff}{\operatorname{Eff}}
\newcommand{\Nef}{\operatorname{Nef}}
\newcommand{\Mov}{\operatorname{Mov}}
\newcommand{\Aut}{\operatorname{Aut}}
\newcommand{\Id}{\operatorname{Id}}
\newcommand{\Bir}{\operatorname{Bir}}
\newcommand{\MCD}{\operatorname{MCD}}

\DeclareMathOperator{\Sec}{Sec}
\DeclareMathOperator{\Jo}{Join}
\DeclareMathOperator{\Cone}{Cone}

\title[On the Fano variety of $(m-1)$-planes in $Q_1\cap Q_2\subset\pr^{2m+2}$]
{On the Fano variety of linear spaces contained in two odd-dimensional quadrics}

\begin{document}

\author[Carolina Araujo]{Carolina Araujo}
\address{\sc Carolina Araujo\\
IMPA\\
Estrada Dona Castorina 110\\
22460-320 Rio de Janeiro\\ Brazil}
\email{caraujo@impa.br}

\author[Cinzia Casagrande]{Cinzia Casagrande}
\address{\sc Cinzia Casagrande\\
Universit\`a di Torino\\
Dipartimento di Matematica\\
via Carlo Alberto 10, 10123 Torino\\ Italy}
\email{cinzia.casagrande@unito.it}

\maketitle

\begin{abstract}
In this paper we describe the geometry of the $2m$-dimensional Fano manifold $G$ parametrizing  $(m-1)$-planes in a smooth
complete intersection $Z$ of two quadric hypersurfaces  in the complex projective space $\pr^{2m+2}$, for $m\geq 1$. 
We show that  there are exactly $2^{2m+2}$ distinct isomorphisms in codimension one between  
$G$ and the blow-up of $\pr^{2m}$ at $2m+3$ general points, parametrized by the $2^{2m+2}$ distinct 
$m$-planes contained in $Z$, and describe these rational maps explicitly.
We also describe the cones of nef, movable and effective divisors of $G$, as well as their dual cones of curves. 
Finally, we determine the automorphism group of $G$. 

These results generalize to arbitrary even dimension the classical description of quartic del Pezzo surfaces ($m=1$).
\end{abstract}

{\small\tableofcontents}


\section{Introduction}
\noindent In this paper we describe the geometry of the $2m$-dimensional Fano manifold $G^{(2m)}$ parametrizing  $(m-1)$-planes in a smooth
complete intersection of two quadric hypersurfaces  in the complex projective space $\pr^{2m+2}$, for $m\geq 1$. 
The case $m=1$ is classical:
\begin{parg} \label{surfaces}
The surface $S=G^{(2)}$ is itself a smooth complete intersection of two quadric hypersurfaces  in $\pr^{4}$, 
and hence a quartic del Pezzo surface.  It is well-known that $\rho(S)=6$, and that 
the cone of effective curves of $S$ is generated by the classes of its $16$ lines.
These $16$ lines have a very special incidence relation: each line intersects properly exactly $5$ lines. 
The del Pezzo surface $S$ can also be described  as the blow-up of $\pr^2$ at $5$ points in general linear position. 
In fact, there are $16$ different ways to realize $S$ as such blow-up: 
For every line $\ell \subset S$, there is a birational morphism $\pi_{\ell}\colon S\to \pr^2$,
unique up to projective transformation of $\pr^2$, contracting the $5$ 
lines incident to $\ell$ to points $p_1^\ell, \dots, p_5^\ell\in \pr^2$ in general linear position. 
The image of $\ell$ under $\pi_{\ell}$ is the unique conic through the $p_i$'s, and the image of the other $10$ lines
are the $10$ lines through $2$ of the  $p_i$'s. Moreover, for any two lines $\ell, \ell' \subset S$,
the sets of points $\{p_1^\ell, \dots, p_5^\ell\}$ and $\{p_1^{\ell'}, \dots, p_5^{\ell'}\}$ are related by a projective transformation of $\pr^2$.

The automorphism group $\Aut(S)$ of $S$ is also well understood (see for instance \cite[\S 8.6.4]{dolgachevbook}). 
In order to describe it, we view $\Pic(S)$ with the intersection product as a unimodular lattice. 
Its primitive sublattice $K_S^{\perp}$ is a $D_5$-lattice. We denote by $W(D_5)$ the Weyl group of automorphisms of this lattice. 
For any  $\zeta\in\Aut(S)$, the induced isomorphism $\zeta^*\colon\Pic(S)\to \Pic(S)$ preserves the intersection product and fixes $K_S$.
This yields an inclusion of groups $\Aut(S)\hookrightarrow W(D_5)\cong (\Z/2\Z)^4\rtimes S_5$, 
whose image contains the normal subgroup  $(\Z/2\Z)^4$.
Moreover, if $S$ is general, then $\Aut(S)\cong (\Z/2\Z)^4$. 
\end{parg}

We will show that the picture described in Paragraph~\ref{surfaces} above generalizes to arbitrary even dimension.
We start by fixing some notation. 
Let $m$ be a positive integer, set $n=2m$, and fix $n+3$ distinct points in $\pr^1$, up to order 
and projective equivalence:
$$
(\lambda_1:1),\dotsc,(\lambda_{n+3}:1)\in \pr^1.
$$
With this fixed data, we introduce the two main characters of this paper, $G^{(n)}$ and $X^{(n)}$:
\begin{parg}[$G^{(n)}$]
Let $Z^{(n)}$ be a smooth complete intersection of the following two quadric hypersurfaces in $\pr^{n+2}$:
$$
Q_1\colon \sum_{i=1}^{n+3}x_i^2=0 \ \text{ and } \ Q_2\colon \sum_{i=1}^{n+3}\lambda_ix_i^2=0.
$$
(Up to projective transformation of $\pr^{n+2}$, any  smooth complete intersection of  two quadric hypersurfaces
can be written in this way - see Section~\ref{prel}.)
Then consider the subvariety $G^{(n)}$ of the Grassmannian $\Gr(m-1,\pr^{n+2})$ parametrizing 
$(m-1)$-planes contained in $Z^{(n)}$. 
It is well known that $G^{(n)}$ is a  smooth $n$-dimensional Fano variety with Picard number $\rho(G^{(n)})=n+4$
(see Section~\ref{Fano} and references therein). 
\end{parg}
\begin{parg}[$X^{(n)}$]
Fix a Veronese embedding $\nu_n\colon \pr^1\hookrightarrow\pr^n$, and set $p_i=\nu_n\big((\lambda_i:1)\big)\in \pr^n$.
The points $p_1, \dots, p_{n+3}$ are in general linear position.
(In fact, this gives a natural correspondence between sets of $n+3$ distinct points in $\pr^1$, up to 
projective equivalence, and  $n+3$ points  in general linear position in $\pr^n$, up to 
projective equivalence.)
Let $X^{(n)}$ be the blow-up of $\pr^n$ at the points $p_1, \dots, p_{n+3}$.
\end{parg}

Our starting point is the following.
\begin{thm}[\cite{bauer,parabolic}]\label{pseudoisom}
The varieties $G^{(n)}$ and $X^{(n)}$ are isomorphic in codimension $1$.
\end{thm}
The proof of Theorem~\ref{pseudoisom} makes use of moduli spaces of parabolic vector bundles.
By \cite{parabolic}, $G^{(n)}$  is isomorphic to the moduli space $\ma{M}^{(n)}$ of stable rank $2$ parabolic vector bundles on 
$(\pr^1,(\lambda_1:1),\dotsc,(\lambda_{n+3}:1))$ of degree zero and weights $(\frac{1}{2},\dotsc,\frac{1}{2})$.
On the other hand, by \cite{bauer} (see also \cite[Theorem 12.56]{mukaibook}), 
$X^{(n)}$ is isomorphic to the moduli space of stable rank $2$ parabolic vector bundles on 
$(\pr^1,(\lambda_1:1),\dotsc,(\lambda_{n+3}:1))$ of degree zero and weights $(\frac{1}{n},\dotsc,\frac{1}{n})$,
and this is  isomorphic to $\ma{M}^{(n)}$ in codimension $1$.

This proof, however, does not give much information about the possible 
isomorphisms  in codimension $1$ between $G^{(n)}$ and $X^{(n)}$.
We call an isomorphism  in codimension $1$ a \emph{pseudo-isomorphism}.
In this paper we describe explicitly the birational maps $G^{(n)}\dasharrow \pr^n$ inducing
a pseudo-isomorphism $G^{(n)}\dasharrow X^{(n)}$.
As we shall see, up to  automorphism of $\pr^n$, there are exactly $2^{n+2}$ distinct such birational maps,  
parametrized by the $2^{n+2}$ linear $\pr^m$'s contained in $Z^{(n)}$.  
In order to state this precisely, we need to recall some  facts about $Z^{(n)}$ (see Section~\ref{prel} and references therein).

The set $\ma{F}_{m}(Z^{(n)})$ of  $m$-planes in $Z^{(n)}$ has cardinality   $2^{n+2}$. 
For each $i=1,\dotsc,n+3$, consider the involution $\sigma_i\colon Z^{(n)}\to Z^{(n)}$  switching the sign of the coordinate $x_i$. 
The group generated by these involutions is isomorphic to $(\Z/2\Z)^{n+2}$, and acts on $\ma{F}_{m}(Z^{(n)})$ freely and transitively.
For every subset $I\subseteq\{1,\dotsc,n+3\}$, we set $\sigma_I:=\prod_{i\in I}\sigma_i=\prod_{j\in I^c}\sigma_j$.
For every $M\in\ma{F}_{m}(Z^{(n)})$ and $I\subset\{1,\dotsc,n+3\}$ with $|I|\leq m+1$, we have
$\dim\left(M\cap\sigma_I(M)\right)=m-|I|$.
Consider the incidence variety $\mathcal{I}:=\{([L],p)\in G^{(n)}\times Z^{(n)}\,|\,p\in L\}$
and the associated diagram
$$
\xymatrix{
&{\mathcal{I}}\ar[dl]_{\pi}\ar[dr]^e&\\
G^{(n)}&&Z^{(n)}.
}
$$
We show that for every $m$-plane $M\in\ma{F}_{m}(Z^{(n)})$, $E_M:=\pi_*(e^*(M))$
is the class of a unique prime divisor  on $G^{(n)}$, which we denote by the same symbol (see Proposition \ref{effdivGMov1G}).

Now we can state our main result. See Theorem~\ref{SQM} for more details, including explicit descriptions 
of the linear systems on $G^{(n)}$ defining  the birational maps $G^{(n)}\dasharrow \pr^n$.
\begin{thm}[Theorem~\ref{SQM} and Corollary~\ref{blowupmodel}]\label{main}
With the notation above, let $M\in\ma{F}_{m}(Z^{(n)})$.
Up to a unique permutation of the $p_i$'s,
there is a unique birational map $\rho_M\colon G^{(n)}\dasharrow\pr^n$,  
inducing a pseudo-isomorphism $G^{(n)}\dasharrow X^{(n)}$,
with the following properties:
\begin{enumerate}[$\bullet$]
	\item The image of $E_M$ under $\rho_M$ is $\Sec_{m-1}(C)$,  the $(m-1)$-th secant variety of the unique rational normal 
		curve $C$  through $p_1,\dotsc,p_{n+3}$ in $\pr^n$. 
	\item The map $\rho_M$ contracts $E_{\sigma_i(M)}$ to the point $p_i\in \pr^n$. 
	\item For each $I\subseteq\{1,\dotsc,n+3\}$ of even cardinality $|I|\leq n$,  
		the image of $E_{\sigma_I(M)}$ under $\rho_M$ is the join of  $\langle p_i\rangle_{i\in I}$ and $\Sec_{s-1}(C)$,
		where $s=\frac{n-|I|}{2}$.
\end{enumerate}

Moreover, any pseudo-isomorphism between $G^{(n)}$ and any blow-up $ \widetilde X$  of $\pr^n$ at $n+3$ points is of this form.
In particular, $\widetilde X\cong X^{(n)}$.
\end{thm}

As immediate corollaries of Theorem~\ref{main}, we obtain the following.
\begin{corollary}\label{punctual}
Let $\ma{P}_1,\ma{P}_2\subset\pr^{n}$ be subsets of $n+3$ distinct points,  and let $X_{\ma{P}_i}$ be the blow-up of $\pr^{n}$ along $\ma{P}_i$,
$i=1,2$. 
Assume that the points in $\ma{P}_1$ are in general linear position.
Then the following are equivalent:
	\begin{enumerate}[$(i)$]
	\item $X_{\ma{P}_1}\cong X_{\ma{P}_2}$;
	\item $X_{\ma{P}_1}$ and $X_{\ma{P}_2}$ are pseudo-isomorphic;
	\item $\ma{P}_1$ and $\ma{P}_2$ are projectively equivalent (as unordered sets).
	\end{enumerate}
\end{corollary}
\begin{corollary}\label{moduli}
Let  $\ma{S}_i=\{(\lambda^i_1:1),\dotsc,(\lambda^i_{n+3}:1)\}\subset  \pr^1$, $i=1,2$, be subsets of $n+3$ distinct  points. 
For each $i\in\{1,2\}$, let $Z_{\ma{S}_i}\subset\pr^{n+2}$ be the smooth complete intersection of the two quadrics:
$$
Q_1\colon \sum_{j=1}^{n+3}x_j^2=0 \ \text{ and } \ Q^i_2\colon \sum_{j=1}^{n+3}\lambda^i_jx_j^2=0,
$$
and let $G_{\ma{S}_i}$ be the variety of $(m-1)$-planes contained in $Z_{\ma{S}_i}$. Then the following are equivalent:
	\begin{enumerate}[$(i)$]
	\item $G_{\ma{S}_1}\cong G_{\ma{S}_2}$;
	\item $G_{\ma{S}_1}$ and $G_{\ma{S}_2}$ are pseudo-isomorphic;
	\item $\ma{S}_1$ and $\ma{S}_2$ are projectively equivalent (as unordered sets).
	\end{enumerate}
\end{corollary}
\noindent Notice that Corollary \ref{punctual} is a classical result,
originally due to Coble (see \cite{dolgort} and references therein).
See also \cite{BHK} for a result related to Corollary \ref{moduli}, in terms of moduli spaces of rank 2 parabolic vector bundles on $\pr^1$.

To prove Theorem~\ref{main}, we determine the nef cone of $G^{(n)}$ explicitly, and then compare it with the 
Mori chamber decomposition of the effective cone of $X^{(n)}$ described in \cite{mukaiADE}.
This decomposition encodes the nef cones of all varieties pseudo-isomorphic to $X^{(n)}$.
In order to determine the cone of effective curves and the nef cone of $G^{(n)}$, we generalize to arbitrary dimension a construction of Borcea \cite{borcea91} in dimension $n=4$.
We define isomorphisms 
$$
H^{2n-2}(G^{(n)},\Z)\stackrel{\alpha}{\la} H^{n}(Z^{(n)},\Z)\stackrel{\beta}{\la} H^2(G^{(n)},\Z)
$$ 
such that,  for every $M\in\ma{F}_{m}(Z^{(n)})$, $\beta(M)=E_M$  
and $\alpha^{-1}(M)$ is the class of a line on the dual $m$-plane $M^*\subset G^{(n)}$.
These  isomorphisms  are dual with respect to the intersection products, i.e., for every $x \in  H^{2n-2}(G^{(n)},\Z)$  and $y \in H^n(Z^{(n)},\Z)$,
$x \cdot \beta(y)=\alpha(x) \cdot y$.
They allow us to describe explicitly special cones of curves and divisors on $G^{(n)}$:
\begin{thm}[Theorem~\ref{NE(G)} and Proposition~\ref{effdivGMov1G}]\label{cones_intro}
Let $\ma{E}\subset H^n(Z,\R)$ be the polyhedral cone generated by the classes $\{M\}_{M\in\ma{F}_{m}(Z)}$, and denote by 
$\ma{E}^{\vee}\subset H^n(Z,\R)$ its dual cone. 
Then $\ma{E}^{\vee}\subset\ma{E}$, and
the cones of nef and effective divisors of $G^{(n)}$ and their dual cones of effective and moving curves satisfy:
$$
\Nef(G^{(n)}) \ = \ \beta(\ma{E}^{\vee})\ \subset \ \beta(\ma{E})\ = \ \Eff(G^{(n)}), \text { and }
$$
$$
\Mov_1(G^{(n)}) = \ \alpha^{-1}(\ma{E}^{\vee})\ \subset \  \alpha^{-1}(\ma{E})\ = \ \NE(G^{(n)}). 
$$
\end{thm}
We give a geometric description of the extremal rays and facets of these cones, and the associated contractions
in Section~\ref{section:conesG}. 
In Proposition~\ref{movG} and its following paragraph, we also describe the cone $\Mov^1(G^{(n)})$ of movable divisors of $G^{(n)}$,
and give a geometric description of the curves corresponding to its facets.

We end this paper by determining the automorphism group of the Fano variety  $G^{(n)}$, generalizing the 
description of the automorphism group of a quartic del Pezzo surface in Example~\ref{surfaces}.
In what follows, we write $W(D_{n+3})$ for the Weyl group of automorphism of a $D_{n+3}$-lattice,
and we denote by the same symbol the involution of  $G^{(n)}$ induced by the involution $\sigma_i$ of  $Z^{(n)}$.
\begin{proposition}[Proposition~\ref{Aut(G)}]
There is  an inclusion of groups 
$$
\Aut(G^{(n)})\hookrightarrow W(D_{n+3})\cong (\Z/2\Z)^{n+2}\rtimes S_{n+3},
$$
whose image contains the normal subgroup  $(\Z/2\Z)^{n+2}$
generated by the involutions $\sigma_i$'s of  $G^{(n)}$.

Moreover, if the points  $(\lambda_1:1),\dotsc,(\lambda_{n+3}:1)\in \pr^1$ are general, then $\Aut(G^{(n)})\cong (\Z/2\Z)^{n+2}$. 
\end{proposition} 
 This paper is organized as follows. 
 Section~\ref{prel} is dedicated to smooth complete intersections $Z\subset \pr^{n+2}$, $n=2m$, of two quadric hypersurfaces 
 in even dimensional projective spaces. 
In particular, we investigate the set $\ma{F}_{m}(Z)$ of  $m$-planes in $Z$, and the cone it spans in $H^n(Z,\R)$.
 In Section~\ref{Fano}, we address the Fano variety $G$ of $(m-1)$-planes in $Z$.
 We construct the isomorphisms $H^{2n-2}(G,\Z)\stackrel{\alpha}{\la} H^{n}(Z,\Z)\stackrel{\beta}{\la} H^2(G,\Z)$, and determine some extremal rays of the cone of effective curves of $G$.
 In Section~\ref{blowup}, we consider the blow-up $X$ of $\pr^n$ at $n+3$ points in general linear position. 
 We describe the Mori chamber decomposition of $\Eff(X)$, following  \cite{mukaiADE} and \cite{bauer}.
 From this we can write the nef cone of $G$ in terms of a natural basis for $\Nu(X)$.
In Section~\ref{relation}, we put together the results from the previous sections to prove Theorem~\ref{main}.
In Section~\ref{section:conesG}, we study cones of curves and divisors in $G$, 
giving a geometric description of their facets and extremal rays. 
In Section~\ref{automorphisms}, we describe the automorphism group of the Fano variety $G$. 

 \

\noindent {\bf Notation and conventions.}
We always work over the field ${\mathbb C}$ of complex numbers. 

Given a subvariety $Z\subset \pr^n$ and a non negative integer $d<n$, we denote by 
$\ma{F}_{d}(Z)$ the closed subset  of the Grassmannian $\Gr(d,\pr^{n})$ parametrizing 
$d$-planes contained in $Z$.

\

\noindent {\bf Acknowledgements.}
We thank Ana-Maria Castravet, Alex Massarenti, Elisa Postinghel and the referee for useful comments and discussions.

Carolina Araujo was partially supported by CNPq and Faperj Research Fellowships, and ICTP Simons Associateship.
This work  started during Carolina Araujo's visit to Universit\`a di Torino; the authors are grateful to INdAM (Istituto Nazionale di Alta Matematica) for the support for this visit. 


\section{Smooth complete intersections of two quadrics}\label{prel}
\noindent In this section we describe the geometry of smooth complete intersections of two quadric hypersurfaces in even dimensional complex projective spaces. 
Many of the results are well known and can be found in \cite[Chapter 3]{reidthesis} or \cite[\S 1]{borcea91}, to which we refer for details and proofs. See also the recent paper \cite{dolgachevduncan} for a study of these complete intersections over  a field of characteristic $2$.

Let $n=2m\geq 2$ be an even integer, and let $Z=Q_1\cap Q_2\subset\pr^{n+2}$ be a smooth complete intersection of two quadric hypersurfaces. 
Up to a projective transformation of $\pr^{n+2}$, we can assume that the quadrics have equations:
\stepcounter{thm}
\begin{equation}\label{Q_1&Q_2}
Q_1\colon \sum_{i=1}^{n+3}x_i^2=0,\qquad Q_2\colon \sum_{i=1}^{n+3}\lambda_ix_i^2=0,
\end{equation}
with $\lambda_i\neq\lambda_j$ if $i\neq j$.
Thus $Z$ is determined by $n+3$ distinct points $(\lambda_1:1),\dotsc,(\lambda_{n+3}:1)\in\pr^1$. 
Acting on these points by permutations and projective automorphisms of $\pr^1$ yields projectively isomorphic varieties $Z\subset\pr^{n+2}$.
\begin{parg}[Involutions and double covers]  \label{sigma&pi}
For each $i=1,\dotsc,n+3$, let $\sigma_i\colon Z\to Z$ be the involution switching the sign of the coordinate $x_i$. 
Then $\sigma_1,\dotsc,\sigma_{n+3}$ commute and have the unique relation $\sigma_1\cdots\sigma_{n+3}=\text{Id}_Z$, 
so they generate a subgroup $W'$ of $\Aut(Z)$ isomorphic to $(\Z/2\Z)^{n+2}$. 
For every subset $I\subseteq\{1,\dotsc,n+3\}$, we set $\sigma_I:=\prod_{i\in I}\sigma_i$. Notice that $\sigma_I=\sigma_{I^c}$.

For each $i=1,\dotsc,n+3$, the projection from the $i$th coordinate point in $\pr^{n+2}$ yields a double cover $\pi_i\colon Z\to Q^{n}$, 
where $Q^{n}\subset\pr^{n+1}$ is the smooth quadric having equation $\sum_{j\neq i}(\lambda_j-\lambda_i)x_j^2=0$, 
where $(x_1:\cdots: \hat x_i: \cdots :x_{n+3})$ are projective coordinates
in $\pr^{n+1}$. The involution associated to this double cover is $\sigma_i$.
\end{parg}
\begin{parg}[The set of $m$-planes in $Z$]  \label{F_m(Z)}
Consider the set $\ma{F}_{m}(Z)$ of  $m$-planes in $Z$. It is a finite set  with cardinality   $2^{n+2}$. 
The group $W'$ generated by the involutions $\sigma_i$'s acts on $\ma{F}_{m}(Z)$ freely and transitively.

For every $M\in\ma{F}_{m}(Z)$ and $I\subset\{1,\dotsc,n+3\}$ with $|I|\leq m+1$ we have
  \stepcounter{thm}
\begin{equation}\label{intersection}
\dim\big(M\cap\sigma_I(M)\big)=m-|I|.
\end{equation}
\end{parg}
\begin{parg}  \label{pi(F_m(Z))}
For each $i=1,\dotsc,n+3$, the double cover $\pi_i\colon Z\to Q^{n}$ induces a map 
$$
\mathcal{F}_{m}(Z)\la \ma{F}_{m}(Q^{n}).
$$
Recall that $\ma{F}_{m}(Q^{n})$ has two connected components $T^{\ph}$ and $T^{\psi}$,
and that two $m$-planes $\Lambda,\Lambda' \subset Q^{n}$ belong to the same connected component if and only if 
$\dim(\Lambda\cap\Lambda')\equiv m\mod 2$ (see for instance \cite[Theorem 1.2(b)]{reidthesis} or \cite[Theorem 22.14]{harris}).

Let $M\in\ma{F}_m(Z)$.
We have $\pi_i(\sigma_i(M))=\pi_i(M)$.
On the other hand, if $j\in\{1,\dotsc,n+3\}\smallsetminus \{i\}$, then $M$ and $\sigma_j(M)$ intersect in codimension one by \eqref{intersection}, 
and the same holds for  $\pi_i(M)$ and $\pi_i(\sigma_j(M))$. 
Therefore $\pi_i(M)$ and $\pi_i(\sigma_j(M))$ belong to different connected components of $\ma{F}_{m}(Q^{n})$. 
In general, if $I\subseteq\{1,\dotsc,n+3\}$ does not contain $i$, then $\pi_i(M)$ and $\pi_i(\sigma_I(M))$ belong 
to the same connected component of $\ma{F}_{m}(Q^{n})$ if and only if $|I|$ is even. 
This shows that the image of $\mathcal{F}_{m}(Z)$ in $\ma{F}_{m}(Q^{n})$ consists of $2^{n+1}$ points, half in each connected component.
\end{parg}
\begin{parg}[The cohomomogy group $H^n(Z,\Z)$] \label{H^n(Z)}
  The cohomomogy group $H^n(Z,\Z)$ is isomorphic to $\Z^{n+4}$, and is generated over $\Z$ by the classes of the $m$-planes in $Z$. 
Moreover $H^n(Z,\Z)$ is a unimodular lattice with respect to the intersection form.

For every $M\in \ma{F}_{m}(Z)$ we denote by the same symbol $M$ the corresponding fundamental class in $H^n(Z,\Z)$.
We denote by $\eta\in H^n(Z,\Z)$ the class of a codimension $m$ linear section of $Z\subset\pr^{n+2}$, so that 
$$
\eta^2=4\quad\text{and}\quad\eta\cdot M=1\text{ for every }M\in\ma{F}_{m}(Z).
$$ 
The sublattice $\eta^{\perp}$ (namely the primitive part $H^n(Z,\Z)_0$) is a $D_{n+3}$-lattice. 
We denote by $W(D_{n+3})$ its Weyl group of automorphisms, which is generated by the reflections in the roots of $\eta^{\perp}$. 
It is the full group of automorphisms of the triple $(H^n(Z,\Z),\cdot,\eta)$, 
and it is isomorphic to $(\Z/2\Z)^{n+2}\rtimes S_{n+3}$.
   
The group $W'\cong (\Z/2\Z)^{n+2}$ generated by the involutions $\sigma_i$'s acts naturally and faithfully on $H^n(Z,\Z)$. 
We still denote by $\sigma_I$ the involution of $H^n(Z,\Z)$ induced by $\sigma_I\colon Z\to Z$.
So we view $W'$ as a subgroup  of $W(D_{n+3})$.
It  is a normal subgroup  with quotient $W(D_{n+3})/W'$  isomorphic to the symmetric group $S_{n+3}$.

For every $M\in\ma{F}_{m}(Z)$ and $i,j\in\{1,\dotsc,n+3\}$ with $i\neq j$ we have
  \stepcounter{thm}
\begin{equation}\label{eta}
\eta=M+\sigma_i(M)+\sigma_j(M)+\sigma_{ij}(M).
\end{equation}
\begin{notation}\label{M_0}
Fix $M_{0}\in \ma{F}_{m}(Z)$. 
For  every $i=1,\dotsc,n+3$, we set  $M_i:=\sigma_i(M_0)$.
More generally, 
for every subset $I\subseteq\{1,\dotsc,n+3\}$, we set  $M_I:=\sigma_I(M_0)$. Notice again that $M_I=M_{I^c}$.
We also set
\stepcounter{thm}
\begin{equation}\label{eps}
\eps_i:=M_0+M_i-\frac{1}{2}\eta\in H^n(Z,\R)\quad\text{for every }i=1,\dotsc,n+3.
\end{equation}
Then $\{\eta,\eps_1,\dotsc,\eps_{n+3}\}$ is an orthogonal basis for $H^n(Z,\R)$, which is useful for computations. 
We  have
\stepcounter{thm}
\begin{equation}\label{norme}
\eta^2=4\quad\text{and}\quad \eps_i^2=(-1)^{m}\text{ for every }i=1,\dotsc,n+3.\end{equation}
In particular, the intersection form on $H^n(Z,\R)$ is positive definite when $n\equiv 0\mod 4$, and has signature $(1,n+3)$ when $n\equiv 2\mod4$. 
Notice that this basis depends on the choice of $M_0$. 
\end{notation}
Let $G_0\subset W(D_{n+3})$ be the stabilizer of $M_0$. Then $G_0\cong S_{n+3}$ and $G_0$ acts by (the same) permutations both on 
$\{M_1,\dotsc,M_{n+3}\}$ and on $\{\eps_1,\dotsc,\eps_{n+3}\}$. We have $W(D_{n+3})=W'\rtimes G_0$. 
Moreover, for every $I\subseteq\{1,\dotsc,n+3\}$ of even cardinality, we have
\stepcounter{thm}
\begin{equation}\label{standardaction}
\sigma_I(\eps_i)=\begin{cases}\eps_i&\text{if }i\not\in I,\\-\eps_i&\text{if }i\in I.\end{cases}
\end{equation}
Thus we see the usual action of $W(D_{n+3})$ on the linear span of $\eps_1,\dotsc,\eps_{n+3}$ by permutation and even sign changes of $\eps_1,\dotsc,\eps_{n+3}$ (see for instance \cite[\S 12.1]{humphreys}).

We collect some identities in  $H^n(Z,\R)$ that we will use in later computations. 
\stepcounter{thm}
\begin{gather}
\label{M_I}
M_I=\frac{1}{4}\eta+\frac{(-1)^{|I|}}{2}\left(\sum_{j\not\in I}\eps_j-\sum_{i\in I}\eps_i\right)\ \  \text{for every }I\subseteq\{1,\dotsc,n+3\}\\
\stepcounter{thm}
\label{M_I2}
\begin{split} 
M_I=\frac{1}{n+1}\left(\left(n+2-|I|\right)\left(\frac{1}{2}\eta - \sum_{i\in I}M_i\right)
+(|I|-1)\sum_{j\in I^{^c}}M_j\right)\\\text{ for every }I\subseteq\{1,\dotsc,n+3\}\text{ with even cardinality}
\end{split}\\
\stepcounter{thm}
\label{eps_i}
\eps_i=\frac{1}{2(n+1)}\eta-\frac{1}{n+1}\sum_{j=1}^{n+3}M_j+M_i\quad\text{for every }i=1,\dotsc,n+3.
\end{gather}
 \end{parg}
 
  Our next goal is to describe  the polyhedral cone $\ma{E}$ in $H^n(Z,\R)$ generated by the classes of $m$-planes in $Z$.
As we shall see below, this is a cone over a $(n+3)$-dimensional \emph{demihypercube}. Before we start discussing the cone $\ma{E}$,
we gather some results about demihypercubes.
\begin{parg}[The demihypercube]\label{demihypercube}
Let $N\geq 4$ be an integer. 
Write $(\alpha_1, \dots, \alpha_{N})$ for coordinates in  $\R^{N}$. 
The vertices of  the hypercube $\left[-\frac{1}{2},\frac{1}{2}\right]^{N}\subset \R^{N}$ are the points of the form 
$v_I=\big((v_I)_1, \dots, (v_I)_{N}\big)$, where 
$I\subseteq \{1, \dots, N\}$, $(v_I)_i=\frac{1}{2}$ if $i\in I$, and  $(v_I)_i=-\frac{1}{2}$ otherwise.
The parity of the vertex $v_I$ is the parity of $|I|$.
For each subset $I\subseteq \{1, \dots, N\}$, define the degree $1$ polynomial in the $\alpha_i$'s:
  \stepcounter{thm}
\begin{equation} \label{eq:H_I}
H_I \ := \ \sum_{j\not\in I} \left(\frac{1}{2}+\alpha_j \right)+ \sum_{i\in I}\left(\frac{1}{2}-\alpha_i\right). 
\end{equation} 
Notice that for any two subsets $I,J\subset  \{1, \dots, N\}$, 
  \stepcounter{thm}
\begin{equation} \label{H_I(v_J)}
H_I (v_J)\ = \ \# (I\smallsetminus J) + \# (J\smallsetminus  I) 
\end{equation} 
is the graph distance of $v_I$ and $v_J$ in the skeleton of the hypercube $\left[-\frac{1}{2},\frac{1}{2}\right]^{N}$.

The  \emph{demihypercube} is the polytope $\Delta \subset \left[-\frac{1}{2},\frac{1}{2}\right]^{N}$
generated by the odd vertices of the hypercube. 
The polytope $\Delta$ has $2^{N-1}+2N$ facets (see for instance \cite[Lemma 2.3]{halfcube}).
More precisely, the polytope $\Delta$ is defined in a minimal way by the following set of inequalities: 
  \stepcounter{thm}
\begin{equation} \label{eq:Delta}
\Delta \ = \  \left\{ 
\begin{aligned}
& -\frac{1}{2}\leq \alpha_i \leq \frac{1}{2}, \ & i\in \{1, \dots, N\} \\
&\ H_I\geq 1 , \ & |I| \text{ even.}
\end{aligned}
\right.
\end{equation} 
Notice that the facets of $\Delta$ supported on the hyperplanes $\left( \alpha_i=\pm \frac{1}{2}\right)$ are isomorphic to the
$(N-1)$-dimensional demihypercube. In particular, they are not simplicial. 
On the other hand, the facet supported on the hyperplane $(H_I=1)$, for $|I|$ even, is
the $(N-1)$-dimensional simplex generated by the $N$ vertices of $\left[-\frac{1}{2},\frac{1}{2}\right]^{N}$
at graph distance $1$ to $v_I$. 

The demihypercube can also be described as a weight polytope of the root system of type $D_N$, see \cite[Example 8.5.13]{green}. 
\end{parg}

Now we go back to $H^n(Z,\R)$ and consider the convex rational polyhedral cone 
$$
\ma{E}\ :=\ \Cone(M)_{M\in\ma{F}_{m}(Z)}\ \subset \ H^n(Z,\R).
$$
It is the cone over the  $(n+3)$-dimensional polytope
$$
\ma{E}_0=\text{Conv}( M )_{M\in\ma{F}_{m}(Z)}
$$
obtained  by  intersecting  $\ma{E}$ with the affine hyperplane $\ma{H}:=\{\gamma\,|\,\gamma\cdot\eta=1\}$.
Note that  the Weyl group $W(D_{n+3})$ preserves $\ma{E}$, $\ma{H}$, and $\ma{E}_0$.

We fix $M_0\in\ma{F}_m(Z)$ and consider the orthogonal basis $\{\eta,\eps_1,\dotsc,\eps_{n+3}\}$ for $H^n(Z,\R)$
introduced in \eqref{eps}. Then $\frac{1}{4}\eta\in\ma{H}$ and 
$\{\eps_1,\dotsc,\eps_{n+3}\}$ is a basis for $\eta^{\perp}$, so that $(\frac{1}{4}\eta,\{\eps_1,\dotsc,\eps_{n+3}\})$
 induces  affine coordinates $(\alpha_1, \dots, \alpha_{n+3})$ on the hyperplane $\ma{H}\cong \R^{n+3}$.\label{alpha}
With these coordinates,  $\frac{1}{4}\eta$ is identified with the origin and, by $\eqref{M_I}$, for every $I\subset\{1,\dotsc,n+3\}$ with $|I|$ even, 
$M_I$ is identified with $v_{I^c}$. Thus
the polytope $\ma{E}_0$ is identified with the
demihypercube $\Delta$ described in Paragraph~\ref{demihypercube}, and $\ma{E}$ with the cone over $\Delta$.
\begin{example}[The surface case]
When $n=2$, $Z\subset \pr^4$ is a smooth quartic del Pezzo surface (see Paragraph~\ref{surfaces}). 
The cone $\ma{E}\subset H^2(Z,\R)$, generated by the classes of the $16$ lines in $Z$, is the cone of effective curves of $Z$. 
In this case the polytope $\ma{E}_0$ is a $5$-dimensional demihypercube, and coincides with the $5$-dimensional Gosset polytope 
(see \cite[\S 8.2.5 and 8.2.6]{dolgachevbook}). In higher dimensions, demihypercubes and Gosset polytopes are different polytopes.
\end{example}

Let us explicitly describe the facets of $\ma{E}$, or equivalently the generators of the dual cone $\ma{E}^{\vee}\subset H^n(Z,\R)$.
Let $(y,x_1,\dotsc,x_{n+3})$ be the coordinates  on $H^n(Z,\R)\cong \R^{n+4}$ induced by the basis $\{\eta,\eps_1,\dotsc,\eps_{n+3}\}$.
It follows from \eqref{eq:Delta} that the cone $\ma{E}$  is defined in a minimal way by the following set of inequalities: 
  \stepcounter{thm}
\begin{equation} \label{eq:E}
\ma{E} \ = \  \left\{ 
\begin{aligned}
& 2y + x_i \geq 0, \ & i\in \{1, \dots, n+3\}, \\
& 2y - x_i \geq 0 , \ & i\in \{1, \dots, n+3\}, \\
& 2(n+1) y +   \sum_{j\not\in I} x_j - \sum_{i\in I} x_i \geq 0, \ & I\subset \{1, \dots, n+3\}  \text{ even.}
\end{aligned}
\right.
\end{equation} 
This is equivalent to saying that the dual cone $\ma{E}^{\vee}\subset H^n(Z,\R)$ is the convex polyhedral cone generated by the classes:
  \stepcounter{thm}
\begin{equation} \label{generators}
 \left\{ 
\small{
\begin{aligned}
& \frac{1}{2}\eta+\eps_i\ \text{ and }\ \frac{1}{2}\eta-\eps_i, \ & i\in \{1, \dots, n+3\}, \\
& \frac{n+1}{2}\eta +   (-1)^{m}\sum_{j\not\in I} \eps_j - (-1)^{m}\sum_{i\in I} \eps_i , \ & I\subset \{1, \dots, n+3\},  \text{ $|I|$ even.}
\end{aligned}
}
\right.
\end{equation} 
\begin{remark}\label{rem:facets}
Using \eqref{eta}, \eqref{eps} and \eqref{M_I}, we can write the generators \eqref{generators} of $\ma{E}^{\vee}$ in terms of $\eta$ and the $M_I$'s:
$$
 \left\{ 
\small{
\begin{aligned}
& \frac{1}{2}\eta+\eps_i = M_0 +M_i,  \\
& \frac{1}{2}\eta-\eps_i = M_j+M_{ij} \ \text{ for any } \ j\neq i, \\
& \frac{n+1}{2}\eta +   (-1)^{m}\sum_{j\not\in I} \eps_j - (-1)^{m}\sum_{i\in I} \eps_i =  2\left(\left\lfloor\frac{m+1}{2}\right\rfloor\eta+(-1)^{m}M_I\right).
\end{aligned}
}
\right.
$$
Note in particular that $\ma{E}^{\vee}\subset\ma{E}$.

For $I\subseteq\{1,\dotsc,n+3\}\smallsetminus \{i\}$, it follows from \eqref{norme} and \eqref{M_I} that:
\stepcounter{thm}
\begin{equation}\label{formula}
\begin{split}
\left(\frac{1}{2}\eta+\eps_i\right)\cdot M_I=&\begin{cases} 1\text{ if }
|I|\equiv m\mod 2,\\
0 \text{ otherwise.}
\end{cases}\\
\left(\frac{1}{2}\eta-\eps_i\right)\cdot M_I=&\begin{cases} 0\text{ if }
|I|\equiv m\mod 2,\\
1 \text{ otherwise.}
\end{cases}
\end{split}
\end{equation}
This describes the generators of the (non-simplicial) facets of $\ma{E}$, corresponding to the extremal rays of $\ma{E}^{\vee}$ generated by $\frac{1}{2}\eta\pm\eps_i$.

For each $M\in\ma{F}_{m}(Z)$, set
$$
\delta_M\ := \ \left\lfloor\frac{m+1}{2}\right\rfloor\eta+(-1)^{m}M.
$$
The  facet of the cone $\ma{E}$ corresponding to the extremal ray 
 of  $\ma{E}^{\vee}$ generated by $\delta_M$ is simplicial, and given by:
$$
\Cone \big(\sigma_i(M)\big)_{i\in \{1, \dots, n+3\}}.
$$
Indeed,  for  
$I\subseteq \{1, \dots, n+3\}$ with $|I|$ odd, one computes using \eqref{M_I}:
$$
\delta_{M}\cdot \sigma_I(M) \ = \  \frac{1}{2}\Big(|I|-1 \Big).
$$ 
\end{remark}

Let $(z,t_1,\dotsc,t_{n+3})$ be the coordinates  on $H^n(Z,\R)$ induced by the basis 
$\big\{\eta,M_1,\dotsc,M_{n+3}\big\}$. In the sequel we need equations for $\ma{E}^{\vee}$ in these coordinates.
Let $I\subseteq\{1,\dotsc,n+3\}$ be such that $|I|\equiv m\mod 2$. 
Using \eqref{M_I},  one computes:
$$
(z\eta+\sum_{i=1}^{n+3}t_iM_i)\cdot M_I \ = \ 
2z+\left(|I|-m\right)\sum_{i=1}^{n+3}t_i-2\sum_{i\in I}t_i.
$$
So we get the following. 
\begin{lemma}\label{dualconeE}
An element $z\eta+\sum_{i=1}^{n+3}t_iM_i$ is in $\ma{E}^{\vee}$ if and only if 
\stepcounter{thm}
\begin{equation}\label{ineqG} 
2z+\left(|I|-m\right)\sum_{i=1}^{n+3}t_i-2\sum_{i\in I}t_i \ \geq \ 0 
\end{equation}
for every  $I\subseteq\{1,\dotsc,n+3\}$ such that $|I|\equiv m\mod 2$.
\end{lemma}

We conclude this section with the following elementary description of the symmetry group of the cone $\ma{E}$.
\begin{lemma}\label{linear}
Let $f\colon H^n(Z,\R)\to H^n(Z,\R)$ be a linear map. The following are equivalent: 
\begin{enumerate}[$(i)$]
\item $f(\ma{E})=\ma{E}$ and $f(x)\cdot\eta=x\cdot\eta$ for every $x\in H^n(Z,\R)$.
\item $f(\ma{E}^{\vee})=\ma{E}^{\vee}$ and $f(\eta)=\eta$.
\item $f\in W(D_{n+3})$.
\end{enumerate}
\end{lemma}
\begin{proof}
The implications $(iii)\Rightarrow (i)$ and $(iii)\Rightarrow (ii)$ are clear.

\medskip

We prove the implication $(i)\Rightarrow (iii)$. Let $f$ be an endomorphism of $H^n(Z,\R)$ satisfying $(i)$.
Then $f$ permutes the vertices of $\ma{E}_0$, and hence  $f(\ma{F}_m(Z))=\ma{F}_m(Z)$.

We fix $M_0\in\ma{F}_m(Z)$ and follow the notation introduced in \ref{M_0}. 
By Remark~\ref{rem:facets},  $\delta_{M_0}=\lfloor\frac{m+1}{2}\rfloor\eta+(-1)^{m}M_0$ generates an extremal ray of $\ma{E}^{\vee}$, 
and the corresponding  facet  of $\ma{E}$ is simplicial given by 
$$
\Cone\big( M_1,\dotsc, M_{n+3}\big ). 
$$
Then $f\big(\Cone\big( M_1,\dotsc, M_{n+3}\big )\big)$ must be another simplicial facet of $\ma{E}$, of the form 
$$
\Cone\big( \sigma_1(M_I),\dotsc, \sigma_{n+3}(M_I)\big )=\sigma_I\big(\Cone\big( M_1,\dotsc, M_{n+3}\big)\big) 
$$ 
for some $I\subseteq\{1,\dotsc,n+3\}$. 
By composing $f$ with the involution $\sigma_I\in  W(D_{n+3})$, we may assume that $f$ fixes the facet $\Cone\big( M_1,\dotsc, M_{n+3}\big )$ of $\ma{E}$.
In particular, $f$ induces a permutation on the set $\big\{ M_1,\dotsc, M_{n+3}\big \}$.
Let $\omega \in  W(D_{n+3})$ be the element in the stabilizer of $M_0$ inducing the same permutation as $f$ 
on the set $\big\{ M_1,\dotsc, M_{n+3}\big \}$.
Then, by composing $f$ with $\omega^{-1}$, we may assume that $f$ fixes each of $M_1,\dots ,M_{n+3}$.

We also have $f(\ma{F}_{m}(Z)\smallsetminus\{M_1,\dotsc,M_{n+3}\})=\ma{F}_{m}(Z)\smallsetminus\{M_1,\dotsc,M_{n+3}\}$, 
therefore $f$ must fix the point
 $$
v:=\sum_{M\in\ma{F}_{m}(Z)\smallsetminus\{M_1,\dotsc,M_{n+3}\}}M.
$$ 
Since $\delta_{M_0}\cdot v>0$,  $v$ is not contained in the linear span of
$M_1,\dotsc,M_{n+3}$ (see Remark~\ref{rem:facets}). This
 implies that $f=\Id_{H^n(Z,\R)}\in W(D_{n+3})$.

\medskip

Finally we prove the implication $(ii)\Rightarrow (iii)$. 
Let $f$ be an endomorphism of $H^n(Z,\R)$ satisfying $(ii)$. 
Then the dual map $g:=f^t\colon H^n(Z,\R)\to H^n(Z,\R)$ satisfies $(i)$, hence by what precedes $g\in W(D_{n+3})$. 
In particular $g$ is orthogonal, and $f=g^t=g^{-1}\in W(D_{n+3})$.
\end{proof}


\section{The Fano variety $G$ of $(m-1)$-planes in $Q_1\cap Q_2\subset\pr^{2m+2}$}\label{Fano}
\noindent Let $n=2m\geq 2$ be an even integer, and let $Z=Q_1\cap Q_2\subset\pr^{n+2}$ be a smooth complete intersection of two quadric hypersurfaces
as in \eqref{Q_1&Q_2}.
In this section we consider the variety $G$ of $(m-1)$-planes in $Z$:
$$
G\ := \ \ma{F}_{m-1}(Z)=\left\{[L]\in\Gr(m-1,\pr^{n+2})\,|\,L\subset Z\right\}.
$$
This is a  smooth $n$-dimensional Fano variety that has been much studied. 
In particular, it is known that 
$\Pic(G)\cong H^2(G,\Z)\cong\Z^{n+4}$, $\Nu(G)\cong H^2(G,\R)$, and $-K_G$ is the restriction of $\ma{O}(1)$ on $\Gr(m-1,\pr^{n+2})$
(see  \cite[Theorem 2.6]{reidthesis}, \cite[Theorem 4.1 and Remark 4.3]{borcea90} and \cite[Proposition 3.2]{jiang}). Moreover $G$ is rational, hence $H^{2n-2}(G,\Z)$ is torsion-free \cite[Proposition 1]{artinmumford} and generated by fundamental classes of one-cycles \cite[Lemma 1]{soulevoisin}. Thus we also have $H^{2n-2}(G,\Z)\cong\Z^{n+4}$ and $\N(G)\cong H^{2n-2}(G,\R)$.

For each $M\in\ma{F}_{m}(Z)$ we set
\stepcounter{thm}
 \begin{equation}\label{M*}
 M^*:=\{[L]\in G\,|\, L\subset M\}.
 \end{equation}
It is an $m$-plane in $G$ (under the Pl\"{u}cker embedding). 
Let $\ell_M\in H^{2n-2}(G,\Z)$ be the class of a line in $M^*$.
By \eqref{intersection}, for every $M,M'\in\ma{F}_m(Z)$ we have:
$$
M^*\cap(M')^*\neq \emptyset\quad \Longleftrightarrow\quad M'=\sigma_i(M)\ \text{ for some }i=1,\dotsc,n+3,
$$
and $M^*\cap\sigma_i(M)^*$ is the point $[M\cap\sigma_i(M)]\in G$.
\begin{parg}[The fibrations $\ph_i$ and $\psi_i$ on $G$]
We define $2(n+3)$ fibrations on $G$, generalizing a construction by Borcea in the case $n=4$  \cite[\S 3]{borcea91}.
For each $i=1,\dotsc,n+3$, the double cover $\pi_i\colon Z\to Q^{n}$ introduced  in Paragraph~\ref{sigma&pi}  induces a map 
$$
\Pi_i\colon G\la \ma{F}_{m-1}(Q^{n}).
$$
Each $(m-1)$-plane  in $Q^{n}$ is contained in exactly one $m$-plane of each of the two families $T^{\ph}$ and $T^{\psi}$
of $m$-planes in $Q^{n}$ (see for instance  \cite[Theorem 22.14]{harris}). This yields two morphisms 
$$
\ma{F}_{m-1}(Q^{n})\to T^{\ph}\subset\Gr(m,\pr^{n+1})\text{ and }\ma{F}_{m-1}(Q^{n})\to T^{\psi}\subset\Gr(m,\pr^{n+1}).
$$ 
By composing them with $\Pi_i\colon G\to\ma{F}_{m-1}(Q^n)$, we get  two distinct morphisms
$$
\bar\ph_i,\bar\psi_i\colon G\la \Gr(m,\pr^{n+1}),
$$
such that $\bar\ph_i(G)\subseteq T^{\ph}$ and $\bar\psi_i(G)\subseteq T^{\psi}$.
Let 
$$
G\stackrel{\ph_i}{\la} Y_{\ph_i}\la \bar\ph_i(G) \quad \text { and } \quad G\stackrel{\psi_i}{\la} Y_{\psi_i}\la \bar\psi_i(G)
$$
be the Stein factorizations of $\bar\ph_i$ and $\bar\psi_i$, respectively. 
\begin{lemma}\label{fibertype}
The morphism $\ph_i\colon G\to Y_{\ph_i}$ has general fiber $\pr^1$, and has exactly $2^n$ singular fibers, each isomorphic to a union 
of two $\pr^m$'s meeting transversally at one point. More precisely, the singular fibers of $\ph_i$ are of the form $M^*\cup\sigma_i(M)^*$, with 
$M\in \mathcal{F}_{m}(Z)$ such that  $[\pi_i(M)]\in T^{\ph}$.
An analogous statement holds for $\psi_i$.

As a consequence, the cone $\NE(\ph_i)$ is the convex cone generated by the classes 
$\ell_M$ for $M\in\ma{F}_m(Z)$ such that $[\pi_i(M)]\in T^{\ph}$, and similarly for $\NE(\psi_i)$. 
\end{lemma}
\begin{proof}
For simplicity we assume in the proof that $m\geq 2$ and $n\geq 4$, the case $n=2$ being classical.

Let $[\Lambda]\in T^{\ph}\subset\Gr(m,\pr^{n+1})$, and let $\Lambda'\subset\pr^{n+2}$ be the $(m+1)$-plane through the $i$th coordinate point 
that projects onto $\Lambda\subset\pr^{n+1}$.
Then $\Lambda'$ is contained in a singular quadric of the pencil of quadrics through $Z$, 
so that $\Lambda'\cap Z=\Lambda'\cap Q_1$  is an $m$-dimensional quadric in $\Lambda'$. 
Hence $[\Lambda]\in\bar\ph_i(G)$ if and only if  $\Lambda'\cap Z$ contains an $(m-1)$-plane.
This happens if and only if the quadric  $\Lambda'\cap Z$ has rank at most $4$.

If  the $m$-dimensional quadric $\Lambda'\cap Z$ has rank $4$, then it is the join of a $(m-3)$-plane with 
a smooth quadric surface $\cong \pr^1\times \pr^1$. 
So it contains two distinct  $1$-dimensional families of $(m-1)$-planes, each parametrized by $\pr^1$. 
Therefore  $\bar\ph_i^{-1}([\Lambda])$ is the disjoint union of two copies of $\pr^1$,
and this yields two smooth fibers of $\ph_i$, each isomorphic to $\pr^1$.

If  $\Lambda'\cap Z$ has rank $3$, then it is the join of an $(m-2)$-plane with a plane conic.
So it contains a one-dimensional family of $(m-1)$-planes, parametrized by the conic. 
Thus in this case $\bar\ph_i^{-1}([\Lambda])_{red}\cong\pr^1$, and  
 this yields a fiber of $\ph_i$ with reduced structure isomorphic to $\pr^1$.

If  $\Lambda'\cap Z$ has rank $2$, then it is the union of two $m$-planes intersecting in codimension one, both projecting onto $\Lambda$. 
Thus there exists $M\in\ma{F}_{m}(Z)$ such that $\Lambda=\pi_i(M)$, $\Lambda'\cap Z= M\cup\sigma_i(M)$, and $\bar\ph_i^{-1}([\Lambda])=M^*\cup\sigma_i(M)^*$. 
It follows from \eqref{intersection} that $M^*$ and $\sigma_i(M)^*$ intersect in one point.

Finally if  $\Lambda'\cap Z$ has rank $1$, then set-theoretically we should have $\Lambda'\cap Z=M$ for some $M\in\ma{F}_{m}(Z)$, and hence $\bar\ph_i^{-1}([\pi_i(M)])=M^*$, which is impossible  because we have already seen that 
$\bar\ph_i^{-1}([\pi_i(M)])=M^*\cup\sigma_i(M)^*$.

Now set 
$$
U:=Y_{\ph_i}\smallsetminus \big\{\ph_i\big(M^*\cup\sigma_i(M)^*\big)\,|\,M\in\ma{F}_m(Z)\text{ and }[\pi_i(M)]\in T^{\ph}\big\}.
$$ 
We have shown that $\ph_i$ has one-dimensional fibers over $U$, and since $G$ is Fano,  $\ph_i$ is a conic bundle over $U$. 
A general singular fiber should be reduced with two irreducible components. Since there are no such fibers,  $\ph_i$ is smooth over $U$.
\end{proof}

In Paragraph~\ref{Y_ph_i} we will characterize the varieties $Y_{\ph_i}$ and $Y_{\psi_i}$. 

Fix $M_0\in \mathcal{F}_{m}(Z)$ such that $[\pi_i(M_0)]\in T^{\psi}$, and 
follow the notation introduced in \ref{M_0}.
It follows from Paragraph~\ref{pi(F_m(Z))} that,  for every $I\subseteq\{1,\dotsc,n+3\}$ such that $i\not\in I$:
$$
[\pi_i(M_I)]\in\begin{cases} T^{\ph}\text{ if } |I| \text{ is odd,}\\
T^{\psi}\text{ if } |I| \text{ is even.}\end{cases}
$$
So we get the following corollary of Lemma~\ref{fibertype}:
\begin{corollary}\label{extremalfaces}
We have:
$$
\NE(\ph_i)=\Cone(\ell_{M_I})_{|I|\text{odd},\,i\not\in I}\quad\text{and}\quad
\NE(\psi_i)=\Cone(\ell_{M_I})_{|I|\text{even},\,i\not\in I}.
$$
The general fiber of $\ph_i$ has class $\ell_{M_j}+\ell_{M_{ij}}$ for $j\neq i$, and the general fiber of $\psi_i$ has class 
$\ell_{M_0}+\ell_{M_i}$.
\end{corollary}
\end{parg}
\begin{parg}[The isomorphisms between $H^{2n-2}(G,\Z)$, $H^n(Z,\Z)$, and $H^2(G,\Z)$]\label{alphabeta}
Recall that, by Poincar\'e duality, the intersection product gives a perfect pairing 
$$
H^2(G,\Z)\times H^{2n-2}(G,\Z)\to\Z.
$$
We will define natural isomorphisms $H^{2n-2}(G,\Z)\cong H^n(Z,\Z)$ and $H^2(G,\Z)\cong H^n(Z,\Z)$, 
which behave well with respect to the intersection products. 
This construction is due to Borcea in the case $n=4$ \cite[\S 2]{borcea91}. 
Throughout this section, we use the same notation as in Section~\ref{prel}.

Consider the incidence variety
$$
\mathcal{I}:=\{([L],p)\in G\times Z\,|\,p\in L\}
$$
and the associated diagram
$$
\xymatrix{
&{\mathcal{I}}\ar[dl]_{\pi}\ar[dr]^e&\\
G&&Z.
}
$$
The morphism $\pi$ is a $\pr^{m-1}$-bundle, hence $\ma{I}$ is smooth, irreducible, of dimension $3m-1=\frac{3}{2}n-1$.
Consider the following morphisms given by pull-backs and Gysin homomorphisms:
\begin{align*} 
\alpha:= & e_*\circ\pi^*\colon H^{2n-2}(G,\Z)\stackrel{\pi^*}{\la} H^{2n-2}(\ma{I},\Z)\stackrel{e_*}{\la} H^n(Z,\Z),\\
\beta:=& \pi_*\circ e^*\colon H^n(Z,\Z)\stackrel{e^*}{\la} H^n(\ma{I},\Z)\stackrel{\pi_*}{\la} H^2(G,\Z),
\end{align*}
so that we have
  \stepcounter{thm}
\begin{equation} \label{alpha&beta}
H^{2n-2}(G,\Z)\stackrel{\alpha}{\la} H^{n}(Z,\Z)\stackrel{\beta}{\la} H^2(G,\Z).
\end{equation}
Note that  $\alpha(\ell_M)=M$ for every $M\in\ma{F}_m(Z)$. 
We set $E_M:=\beta(M)\in H^2(G,\Z)$ for every  $M\in\ma{F}_{m}(Z)$.
\begin{proposition}[\cite{borcea91}, Proposition 2.2]\label{iso}
Both $\alpha$ and $\beta$ are isomorphisms, and they are dual to each other with respect to the intersection products. Namely:
$$
x \cdot \beta(y)=\alpha(x) \cdot y\quad\text{for every } x \in  H^{2n-2}(G,\Z) \text{ and }y \in H^n(Z,\Z).
$$
\end{proposition}
\begin{proof}
Since $\alpha(\ell_M)=M$, and the classes $\{M\}_{M\in\ma{F}_{m}(Z)}$ generate $H^n(Z,\Z)$, the homomorphism $\alpha$ is surjective. 
Then  $\alpha$ must be an isomorphism, because  $H^{2n-2}(G,\Z)$ and $H^n(Z,\Z)$ are free of the same rank. 

It follows from properties of Poincar\'e duality that $\alpha^t=(e_*\circ\pi^*)^t=(\pi^*)^t\circ (e_*)^t=\pi_*\circ e^*=\beta$, 
so $\alpha$ is the transpose homomorphism of $\beta$. It follows that
 $\beta$ must be an isomorphism too. 
\end{proof}
\begin{corollary}
We have  $\beta(\eta)=-K_G$.
\end{corollary}
\begin{proof}
Using Proposition \ref{iso}, for every $M\in\ma{F}_m(Z)$ we have 
$$1=\eta\cdot M=\eta\cdot\alpha(\ell_M)=\beta(\eta)\cdot\ell_M=-K_G\cdot\ell_M.$$
 Since $\alpha$ is an isomorphism, and the classes  $\{M\}_{M\in\ma{F}_{m}(Z)}$ generate $H^n(Z,\Z)$, the classes $\{\ell_M\}_{M\in\ma{F}_{m}(Z)}$ generate $H^{2n-2}(G,\Z)$. This yields the statement.
\end{proof}
Consider the  involution $\sigma_I\colon Z\to Z$, $I\subseteq\{1,\dotsc,n+3\}$  defined in Paragraph~\ref{sigma&pi}.
It  induces an involution of $G$, which we denote by the same symbol:
$$
\sigma_I\colon G\la G,\quad [L]\mapsto [\sigma_I(L)].
$$
Therefore the group $W'\cong (\Z/2\Z)^{n+2}$ generated by the involutions $\sigma_i$'s 
acts on $G$, $H^2(G,\Z)$ and $H^{2n-2}(G,\Z)$.
It also acts on the incidence variety $\ma{I}$ in such a way that
both morphisms $\pi$ and $e$ are $W'$-equivariant.
It follows that the isomorphisms $\alpha$ and $\beta$ are $W'$-equivariant.
\end{parg}
\begin{proposition}\label{ell_M}
For every $M\in\ma{F}_m(Z)$, $\ell_M$ generates an extremal ray of $\NE(G)$.
\end{proposition}
\begin{proof}
Fix $M_0\in \mathcal{F}_{m}(Z)$ and $i\in\{1,\ldots, n+3\}$ such that $[\pi_i(M_0)]\in T^{\psi}$, and 
follow the notation introduced in \ref{M_0}. By Corollary \ref{extremalfaces}, we have:
$$
\alpha(\NE(\ph_i))=\Cone(M_I)_{|I|\text{odd},\,i\not\in I}\quad\text{and}\quad
\alpha(\NE(\psi_i))=\Cone(M_I)_{|I|\text{even},\,i\not\in I}.
$$
By \eqref{formula}, these are facets of the cone $\ma{E}\subset H^n(Z,\R)$, whose extremal rays are generated by the classes 
$M=\alpha(\ell_{M})$ contained in these facets. 
Thus, for every $M\in\ma{F}_m(Z)$ the class $\ell_{M}$ generates an extremal ray of either 
 $\NE(\ph_i)$ or  $\NE(\psi_i)$, and hence of $\NE(G)$. 
\end{proof}


\section{The blow-up $X$ of $\pr^n$ at $n+3$ points}\label{blowup}
\noindent 
Let $n\geq 3$ be an integer. 
Unless otherwise stated, in this section we do not assume that $n$ is even.
Let $\ma{P}=\{p_1,\dotsc,p_{n+3}\}\subset \pr^n$ be a set of distinct points in general linear position, 
and denote by $C$ the unique rational normal curve in $\pr^n$
through these points.  
Let $X=X_{\ma{P}}$ be the blow-up of $\pr^n$ at $p_1,\dotsc,p_{n+3}$. 
Notice that acting on $\ma{P}=\{p_1,\dotsc,p_{n+3}\}$ by permutations and projective automorphisms of $\pr^n$ yields isomorphic varieties $X_{\ma{P}}$.
The variety $X$ and its birational geometry have been widely studied.
We refer the reader to \cite{dolgachev,bauer,mukaiXIV,mukaiADE,CoxCT,araujo_massarenti,BDP} and references therein. 

We have $\Pic(X)\cong H^2(X,\Z)$ and $\Nu(X)\cong H^2(X,\R)$.
We denote by $H$ the pullback to $X$ of the hyperplane class in $\pr^n$, and by $E_i$ the exceptional divisor over the point $p_i$
(as well as its class in $H^2(X,\Z)$). 
\begin{parg}[Special subvarieties of $X$] \label{J_Is}
Given a subset $I\subset \{1, \cdots, n+3\}$, with $|I|=d\leq n$, and an 
integer $0\leq s\leq \frac{n-d}{2}$, we consider the join
$$
\Jo\big(\langle p_i\rangle_{i\in I}, \Sec_{s-1}(C)\big)\subset \pr^n 
$$
(here we write $\Sec_k(C)$ for the subvariety of $\pr^n$ obtained as the closure of the union of all $k$-planes 
spanned by $k+1$ general points of $C$, for $k\geq 0$; in particular
$\Sec_{0}(C)=C$.
 We also set $\Sec_{-1}(C)=\emptyset$.)

This join has dimension equal to $d+2s-1$. 
We denote by $J_{I,s}\subset X$ the strict transform of $\Jo\big(\langle p_i\rangle_{i\in I}, \Sec_{s-1}(C)\big)$.
When $d+2s=n$ (so that $|I^{c}|= n+3-3=2s+3$ is odd) we denote the divisor $J_{I,s}$ and
its class in $H^2(X,\Z)$ by $E_I$; in particular, for $n=2m$ even, $E_{\emptyset}=J_{\emptyset,m}$ is the strict transform of $\Sec_{m-1}(C)$.
For $I={\{i\}^{c}}$, we set $E_I=E_i$. 
For every $I\subset \{1, \cdots, n+3\}$ with $|I^{c}|= 2s+3$ odd, $s\geq 0$, we have the following identity in 
$H^2(X,\Z)$:
\stepcounter{thm}
\begin{equation}\label{E_I}
E_I \ = \ (s+1)H - (s+1)\sum_{i\in I}E_i -s \sum_{j\not\in I}E_j.
\end{equation}
\end{parg}

By \cite[Theorem 1.2]{CoxCT}, each $E_I$ generates an extremal ray of $\Eff(X)$, and all extremal rays
are of this form.
Moreover, by \cite[Theorem 1.3]{CoxCT} and \cite{mukaiADE}, $X$ is a Mori dream space (MDS for short).
We refer to \cite{hukeel} for the definition and basic properties of MDS's.
Here we only recall an important feature of a MDS, the Mori chamber decomposition of its effective cone.
\begin{parg}[The Mori chamber decomposition]\label{MCD}
Let $Y$ be a projective, normal and $\Q$-factorial MDS.
The effective cone $\Eff(Y)$ admits a fan structure,  called \emph{Mori chamber decomposition} and denoted by $\MCD(Y)$,
which can be described as follows (see \cite[Proposition 1.11(2)]{hukeel} and \cite[Section 2.2]{okawa_MCD}).
There are finitely many birational contractions (i.e., birational maps whose inverses do not contract any divisor)
from $Y$ to  projective, normal and $\Q$-factorial MDS's, denoted by $g_i\colon Y\map Y_i$.
The set $\Exc(g_i)$ of classes of exceptional prime divisors of   $g_i$ has cardinality $\rho(Y)-\rho(Y_i)$.
The maximal cones $\ma{C}_i$ of the fan $\MCD(Y)$ are of the form:
$$
\ma{C}_i \ = \ \Cone \ \Big( \ g_i^*\big(\Nef(Y_i)\big)\ , \  \Exc(g_i) \ \Big).
$$
By abuse of notation, we often write $\Nef(Y_i)\subset \Eff(Y)$ for $ g_i^*\big(\Nef(Y_i)\big)\subset \Eff(Y)$.
If $\Exc(g_i)=\emptyset$, then we say that $g_i\colon Y\map Y_i$ is a small $\Q$-factorial modification of $Y$.
The movable cone $\Mov(Y)$ of $Y$ is the union 
$$
\Mov(Y) \ = \ \bigcup_{\Exc(g_i)=\emptyset}\ma{C}_i.
$$

An arbitrary cone $\sigma\in \MCD(Y)$ is of the form
$$
\sigma \ = \ \Cone \ \Big( \ f^*\big(\Nef(W)\big)\ , \  \ma{E} \ \Big),
$$
where $f\colon Y\map W$ is a dominant rational map to  a normal projective variety, which factors as 
$Y\stackrel{g_i}{\dasharrow} Y_i\stackrel{f_i}{\la} W$ for some $i$, where $f_i
\colon Y_i\to W$ is the contraction 
of an extremal face of $\Nef(Y_i)$, and $\ma{E}\subset \Exc(g_i)$.

Given an effective divisor $D$ on $Y$, its class in $\Nu(Y)$ lies in the relative interior of 
some cone in $\MCD(Y)$, say $\Cone \big(  f^*\big(\Nef(W)\big) ,   \ma{E}  \big)$.
The map $f\colon Y\map W$ coincides with  the map $\varphi_{|mD|}$ for $m\gg 1$ divisible enough.
In this case, we write $Y_D$ for the variety $W$.
\end{parg}

Now we go back to $X$.
Our next goal is to describe the Mori chamber decomposition of $\Eff(X)$, following  \cite{mukaiADE} and \cite{bauer}
(see also \cite[Section 3]{araujo_massarenti}).

Let $(y, x_1,\dots, x_{n+3})$ be the coordinates in  $H^2(X,\R)$ induced by the basis  $(H,E_1,\dotsc,E_{n+3})$,
and consider the  affine hyperplane 
$$
\ma{H}=\Big((n+1)y +\sum x_i=1\Big)\subset H^2(X,\R).
$$
It contains all the generators $E_I$ of $\Eff(X)$ described above, as well as  $\frac{1}{4}(-K_X)$.

We now observe that the convex hull of the $E_I$'s in $\ma{H}$ is a demihypercube. To see this, we need suitable coordinates in $\ma{H}$.
For $i=1,\dotsc,n+3$, set 
\stepcounter{thm}
\begin{equation} \label{epstilde}
\tilde\eps_i:=\frac{1}{2}\Big(H-\sum_{j\neq i}E_j+E_i\Big).
\end{equation}
Then $\{\tilde\eps_1,\dotsc,\tilde\eps_{n+3}\}$ is a basis for the linear subspace $\big((n+1)y +\sum x_i=0\big)$, 
so that $\big(\frac{1}{4}(-K_X),\{\tilde\eps_1,\dotsc,\tilde\eps_{n+3}\}\big)$ induces affine coordinates $(\alpha_1,\dotsc,\alpha_{n+3})$ 
in $\ma{H}\cong\R^{n+3}$. The radial projection 
$$
H^2(X,\R)\smallsetminus\Big((n+1)y +\sum x_i=0\Big)\la\ma{H}
$$ 
is given in coordinates by:
\stepcounter{thm}
\begin{equation} \label{eq:phi}
\alpha_i   \ = \  \frac{y+x_i}{(n+1)y +\sum x_i} - \frac{1}{2},\quad\text{ for }i=1,\dotsc,n+3.
\end{equation}
In the coordinates $\alpha_i$,  $\frac{1}{4}(-K_X)$ is identified with the origin,
and $E_I$ with $v_{I^c}$, with  the notation introduced in Paragraph~\ref{demihypercube}. Thus $\Eff(X)\cap\ma{H}$ is identified with the demihypercube 
 $\Delta \subset \R^{n+3}$
described in Paragraph~\ref{demihypercube}:
$$
\Delta \ = \  \left\{ 
\begin{aligned}
& -\frac{1}{2}\leq \alpha_i \leq \frac{1}{2}, \ & i\in \{1, \dots, N\} \\
&\ H_I\geq 1 , \ & |I| \text{ even.}
\end{aligned}
\right.
$$

Recall the degree $1$ polynomials $H_I$ introduced in \eqref{eq:H_I}, and 
consider the hyperplane arrangement:
 \stepcounter{thm}
\begin{equation} \label{eq:MCD_Delta}
  \Big(\ H_I \ = \  k \ \Big)_{\ I\subset \{1, \dots, n+3\}, \ k\in\mathbb{N},  \ 2\leq k\leq \frac{n+3}{2}, \ |I|\not\equiv k \mod 2.}
\end{equation} 
It defines a subdivision of $\Delta$ in polytopes, and a fan structure on $\Eff(X)$, given by the cones over these polytopes.
By  \cite{mukaiADE} and \cite{bauer}, this fan coincides with $\MCD(X)$. 
Moreover, one has the following description of the wall crossings 
(see \cite[Propositions 2 and 3]{mukaiADE} and also \cite[Section 2]{bauer}):
\begin{enumerate}[(1)]
	\item The intersection of $\Mov(X)$ with the hyperplane $\ma{H}$ is given by 
$$
		\Delta_{\text{\em Mov}} \ = \ \Mov(X)\cap\ma{H} \ = \  \left\{ 
		\begin{aligned}
		& -\frac{1}{2}\leq \alpha_i \leq \frac{1}{2}, \ & i\in \{1, \dots, n+3\} \\
		&H_I\geq 2 , \ & |I| \text{ odd.}
		\end{aligned}
		\right.
$$
	\item All small $\Q$-factorial modifications of $X$ are smooth.
	\item  Let $\ma{C}$ be a maximal cone of $\MCD(X)$, contained in $\Mov(X)$, corresponding to a small $\Q$-factorial modification $\widetilde X$ of $X$.
		Let  $\sigma\subset \partial \ma{C}$ be a wall such that $\sigma\subset \partial \Mov(X)$, and let 
		$f\colon \widetilde X\to Y$ be the corresponding elementary contraction.
		Then $\sigma\cap\ma{H}\subset\Delta_{\text{\em Mov}}$ is supported on a hyperplane of one of the following forms:
		\begin{enumerate}
			\item ($\alpha_i=-\frac{1}{2}$) or ($\alpha_i=\frac{1}{2}$). 
			\item ($H_I= 2$), with $|I|$ odd. 
		\end{enumerate}
		In case (a),  $f\colon \widetilde X\to Y$ is a $\pr^1$-bundle. In case (b), $f\colon \widetilde X\to Y$ is the blow-up of a smooth point,
		and the exceptional divisor of $f$ is the strict transform  in $\widetilde X$ of the divisor $E_{I^{^c}}\subset X$.
	\item Let $\ma{C}$ and $\ma{C}'$ be two maximal cones of $\MCD(X)$, contained in $\Mov(X)$, and having a common facet. 
		Let $f\colon X\map \widetilde X$ and $f'\colon X\map\widetilde X'$  be the corresponding 
	 	small $\Q$-factorial modifications of $X$.
		The intersections of these cones with $\ma{H}$ are separated in $\Delta$ by a hyperplane of the form
		$(H_I =  k)$, with $3\leq k\leq \frac{n+3}{2}$ and $|I|\not\equiv k \mod 2$. 
		Suppose that $\ma{C}\cap\ma{H}\subset (H_I \leq  k)$ and $\ma{C}'\cap\ma{H}\subset (H_I \geq  k)$.
		Then the birational map $f'\circ f^{-1}\colon \widetilde X \map \widetilde X'$ flips a $\pr^{k-2}$ into a $\pr^{n+1-k}$.
\end{enumerate}
\begin{remark}\label{flip_locus}
It is possible to give a more precise description of the flipping locus $\pr^{k-2}\subset \widetilde X$ (or $\pr^{n+1-k}\subset \widetilde X'$) 
in the situation described under (4) above (see \cite[Proposition 2.6(iv) and Theorem 2.9]{bauer}):
Consider the nef cone of $X$ and its section with $\ma{H}$, 
$$
		\Delta_{\text{\em Nef}} \ = \ \Nef(X)\cap\ma{H} \ = \  \left\{ 
		\begin{aligned}
		&H_{\{i\}}\geq 2, \ & i\in \{1, \dots, n+3\} \\
		&H_{\{i,j\}}\leq 3 , \ &  i, j\in \{1, \dots, n+3\}, \ i\neq j.
		\end{aligned}
		\right.
$$

Suppose that $\Delta_{\text{\em Nef}}\subset  (H_I \leq  k)$. 
Then the $\pr^{k-2}\subset \widetilde X$ flipped by $f'\circ f^{-1}$ is the strict transform in  $\widetilde X$ of the special variety  $J_{I,s}\subset X$,
where $s=\frac{k-|I|-1}{2}\geq 0$.

Suppose that $\Delta_{\text{\em Nef}}\subset  (H_I \geq  k)$. 
Then the $\pr^{n+1-k}\subset \widetilde X'$ flipped by $f\circ (f')^{-1}$ is the strict transform  in $\widetilde X'$ of the special variety $J_{I^c,s'}\subset X$,
where $s'=\frac{|I|-k-1}{2}\geq 0$. 
\end{remark}
\begin{remark}\label{projection}
Recall from Paragraph \ref{demihypercube} the description of the facets of $\Delta$. 
Each of the $2(n+3)$ facets of $\Delta$ supported on the hyperplanes $(\alpha_i=\pm\frac{1}{2})$ intersects $\Delta_{\text{\em Mov}}$ along a facet, 
while the other facets of $\Delta$, supported on the hyperplanes $(H_I=1)$ for $|I|$ even, are disjoint from $\Delta_{\text{\em Mov}}$.
Let us describe the rational maps associated to the facets of $\Delta_{\text{\em Mov}}$ supported on the hyperplanes $(\alpha_i=\pm\frac{1}{2})$.

Fix $i\in\{1,\dotsc,n+3\}$ and let
$\ma{P}_i\subset \pr^{n-1}$ be the image of the set $\ma{P}\smallsetminus \{p_i\}$ under the projection 
$\pi_{p_i}\colon \pr^n\map \pr^{n-1}$  from $p_i$. Let $Y=(X_{\ma{P}_i})^{n-1}$ be the blow-up of $\pr^{n-1}$ at the $n+2$ points in $\ma{P}_i$. 

There is a small $\Q$-factorial modification $X\map X_i$ and a $\pr^1$-bundle $X_i\to Y$
extending $\pi_{p_i}$  (see \cite[Example 1]{mukaiADE}).  Let $\pi_i\colon X\dasharrow Y$ be the composite map.
The general fiber of $\pi_i$ is the strict transform in $X$ of a general line in $\pr^n$ through $p_i$.
The hyperplane $(\pi_i)^*H^2(Y,\R)$ has equation $y+x_i=0$. Using \eqref{eq:phi}, we see that  $(\pi_i)^*H^2(Y,\R)\cap\ma{H}$ is the hyperplane 
$\big(\alpha_i=-\frac{1}{2}\big)$. Thus the cone  $(\pi_i)^*\Eff(Y)$ is the cone over  the polytope $\Delta\cap(\alpha_i=-\frac{1}{2})$, which is an $(n+2)$-dimensional demihypercube.

Similarly, there is a map $\pi'_i\colon X\dasharrow Y$ whose general fiber is the strict transform in $X$ of a general rational normal curve through 
the points $p_\lambda$, $\lambda\neq i$. 
Indeed, fix $j\neq i$ and let $\varphi:\pr^n \map \pr^n$ be the standard Cremona transformation centered at the points $p_\lambda$, $\lambda\neq i,j$. 
This map sends rational normal curves through the points $p_\lambda$, $\lambda\neq i$, to lines through $\varphi(p_j)$. 
There is an automorphism of $\pr^n$ fixing $p_\lambda$, $\lambda\neq i,j$, sending $p_j$ to $\varphi(p_i)$, and  sending $p_i$ to $\varphi(p_j)$
(see Remark~\ref{rem_Dol}).
By composing $\varphi$ with the projection from $\varphi(p_j)$, we obtain a rational map $\pi'_{p_i}\colon \pr^n\map Y$ 
whose general fiber is a general rational normal curve through the points $p_\lambda$, $\lambda\neq i$. 
This yields a $\pr^1$-bundle $X'_i\to Y$ on a small $\Q$-factorial modification of $X$, and the desired map $\pi'_i\colon X\dasharrow Y$.
As before, one checks that $(\pi_i)^*\Eff(Y)$ is the cone over  the demihypercube $\Delta\cap(\alpha_i=\frac{1}{2})$.
\end{remark}

The center of the polytopes $\Delta_{\text{\em Mov}}$ and $\Delta$ is the origin $\bar 0\in\R^{n+3}$, which corresponds to
$\frac{1}{4}(-K_X)$. In particular, the divisor $-K_X$ is movable.
We want to describe the Fano model $X^n_{\text{\em Fano}}:= X_{-K_X}$.

If $n$ is odd, then $\bar 0$ is a vertex in the subdivision of $\Delta$, and is
contained in the intersection of the hyperplanes:
$$
 \left( \ H_I\ = \  \frac{n+3}{2} \ \right)_{|I|\not\equiv \frac{n+3}{2} \mod 2}.
$$
Thus $-K_X$ lies in a one-dimensional cone of the fan $\MCD(X)$, contained in the interior of $\Mov(X)$.
Therefore $X^n_{\text{\em Fano}}$ is non $\Q$-factorial and has  Picard number $1$. 

For the remaining of this section, we assume that $n=2m\geq 2$ is even.
Then $\bar 0$ lies in the interior of a maximal polytope in the subdivision of $\Delta_{\text{\em Mov}}$, namely
the polytope defined by: 
 \stepcounter{thm}
\begin{equation} \label{Fano_chamber}
\Delta_{\text{\em Fano}} \ = \ \big( \ H_I\ \geq \  m+1 \ \big)_{|I|\equiv m \mod 2}.
\end{equation}
Then $X^n_{\text{\em Fano}}$ is a small $\Q$-factorial modification of $X$, it is a smooth Fano manifold, and 
$\Nef(X^n_{\text{\em Fano}})\subset \Eff(X)$ is the cone over the polytope $\Delta_{\text{\em Fano}}$. 
 \begin{remark}\label{G=X_Fano}
By
Theorem~\ref{pseudoisom}, when $\ma{P}$ is the image of $\big\{(\lambda_1:1),\dotsc,(\lambda_{n+3}:1)\big\}\subset \pr^1$
under a Veronese embedding $\pr^1\hookrightarrow\pr^n$,  $X$ is pseudo-isomorphic to the Fano variety $G$ addressed in Section~\ref{Fano}. This implies that $X^n_{\text{\em Fano}}$ is isomorphic to $G$.
\end{remark}
\begin{parg}\label{X_Fano}
Using the properties of MDS's, and the description of $\MCD(X)$ above, we can deduce many properties of 
$X^n_{\text{\em Fano}}$:
\begin{enumerate}[$\bullet$]
\item The Mori cone $\NE(X^n_{\text{\em Fano}})$ admits 
exactly $2^{n+2}$ extremal rays, whose corresponding contractions all contract a $\pr^{m}$
to a point.
\item The variety $X^n_{\text{\em Fano}}$ admits $2(n+3)$ distinct (non-trivial) contractions of fiber type.
Indeed, the points in $\partial \Delta_{\text{\em Mov}}\cap \Delta_{\text{\em Fano}}$ are those of the form $\alpha=(\alpha_1, \dots, \alpha_{n+3})$,
where $\alpha_i = -\frac{1}{2} \text{ or } \frac{1}{2}$ for some fixed $i$, and  $\alpha_j =0$ for $j\neq i$. 
These points all lie in $\partial \Delta$.
We denote the corresponding contractions by $\phi_i$ and $\phi'_i$,
respectively. 
\end{enumerate}
\end{parg}
\begin{lemma} \label{image_phi}
The  morphisms $\phi_i$ and $\phi'_i$ are generic $\pr^1$-bundles over $(X_{\ma{P}_i})^{n-1}_{\text{\it Fano}}$, 
where  $\ma{P}_i\subset \pr^{n-1}$ is as in Remark \ref{projection}.
The general fiber of $\phi_i$ is the strict transform in $X^n_{\text{\it Fano}}$ of a general line in $\pr^n$ through $p_i$.
The general fiber of $\phi'_i$ is the strict transform in $X^n_{\text{\it Fano}}$ of a general rational normal curve  in $\pr^n$ through
 $\ma{P}\smallsetminus \{p_i\}$. 
\end{lemma}
\begin{proof}
Let $\alpha=(\alpha_1, \dots, \alpha_{n+3})$,
where $\alpha_i = -\frac{1}{2}$  and  $\alpha_j =0$ for $j\neq i$, 
and consider the corresponding fibration  $\phi_i\colon X^n_{\text{\em Fano}}\to X_D$,
where $D$ is an effective divisor such that $\R_{\geq 0}[D]\cap\ma{H}=\alpha$.

Consider the map $\pi_i\colon X\dasharrow Y:=(X_{\ma{P}_i})^{n-1}$ introduced in Remark 
\ref{projection}, and recall that $(\pi_i)^*\Eff(Y)$ is the cone over the $(n+2)$-dimensional demihypercube $\Delta\cap(\alpha_i=-\frac{1}{2})$. The center of this demihypercube is $\alpha$, hence $D$ is a positive multiple of 
 $(\pi_i)^*(-K_Y)$.
So the image $X_D$ of $\phi_i$ is precisely the Fano model $(X_{\ma{P}_i})^{n-1}_{\text{\em Fano}}$ of $Y$.

A similar argument shows the statement for $\phi'_i$.
 \end{proof}
 \begin{parg}
Let $(z,t_1,\dotsc,t_{n+3})$ be new coordinates in  $H^2(X,\R)$, induced by the basis  $\{-K_X,E_1,\dotsc,E_{n+3}\}$. 
These are related to  $(y, x_1,\dots, x_{n+3})$ by  $y=z(n+1)$ and $x_i=t_i-(n-1)z$.
Using the defining inequalities for $\Delta_{\text{\em Fano}}$ in 
\eqref{Fano_chamber}, and the 
expression for the radial projection onto $\ma{H}$ in \eqref{eq:phi}, we conclude that 
$\Nef(X^n_{\text{\em Fano}})\subset H^2(X,\R)$  is defined by the inequalities:
 \stepcounter{thm}
 \begin{equation}\label{fanochamber}
 2z+\left(|I|-m\right)\sum_{i=1}^{n+3}t_i-2\sum_{i\in I}t_i\geq 0
 \end{equation}
 for every $I\subseteq\{1,\dotsc,n+3\}$ such that $|I|\equiv m\mod 2$. 
\end{parg}
\begin{parg}\label{X-->X_Fano}
We end this section by describing the birational map $X\map X^n_{\text{\em Fano}}$.
First notice that to go from the interior of the polytope $\Delta_{\text{\em Nef}}  =  \Nef(X)\cap\ma{H}$ to the interior of the polytope
$\Delta_{\text{\em Fano}} =  \Nef(X^n_{\text{\em Fano}})\cap\ma{H}$,
we must cross the wall $ \big(H_I  =  k  \big)$ for every $I\subset \{1, \dots, n+3\}$
and $3\leq k\leq m+1$ such that $|I|\not\equiv k \mod 2$ and $|I|\leq k-1$.
By Remark~\ref{flip_locus} and \cite[Theorem 2.9]{bauer}, we conclude that the rational map 
$X\map X^n_{\text{\em Fano}}$ factors as:
$$
X=X_0\stackrel{\varphi_1}{\map} X_1\stackrel{\varphi_2}{\map}X_2 \map \cdots \stackrel{\varphi_{m-1}}{\map}X_{m-1}=X^n_{\text{\em Fano}},
$$
where each $\varphi_i\colon X_{i-1}\map X_i$ flips the strict transforms in $X_{i-1}$ of all special subvarieties $J_{I,s}\subset X$
of dimension $i$. These strict transforms are disjoint in $ X_{i-1}$ and each isomorphic to $\pr^i$. 
The flipped locus on $X_i$ is a disjoint union of $\pr^{n-1-i}$'s, one for each $J_{I,s}$ of dimension $i$. Notice that in general the map $\varphi_i$ is not the flip of a small contraction: it is a pseudo-isomorphism that can be factored as a sequence of flips.

In particular, we can describe the $2^{n+2}$  $\pr^{m}$'s in $X^n_{\text{\em Fano}}$ corresponding to the $2^{n+2}$ extremal rays
of $\NE(X^n_{\text{\em Fano}})$. These are the strict transforms of the special subvarieties $J_{I,s}\subset X$ of dimension $m$, and the flipped locus 
of the flips of the strict transforms of the special subvarieties $J_{I,s}\subset X$ of dimension $m-1$. These are, respectively:
{\small\begin{align*}
\text{$m$-dimensional $J_{I,s}$: }
\quad &\sum_{\substack{d=0\\d\not\equiv m\text{ mod } 2}}^{m+1}\binom{n+3}{d}\\
\text{$(m-1)$-dimensional $J_{I,s}$: }
\quad &\sum_{\substack{d=0\\d\equiv m\text{ mod }2}}^{m}\binom{n+3}{d}.
\end{align*}}

We can also describe the strict transforms in $X^n_{\text{\em Fano}}$ of the divisors $\pr^{n-1}\cong E_i\subset X$ under the rational map 
$X\map X^n_{\text{\em Fano}}$.
There are $n+3$ special points $q_1, \dotsc, q_{n+3}\subset E_i$: $q_j$ is the intersection of $E_i$ with the strict transform of the line 
through $p_i$ and $p_j$ when $j\neq  i$, and $q_i$ is the intersection of $E_i$ with the strict transform of  $C$. The points $q_i$'s all lie in a 
rational normal curve $C'$ of degree $n-1$ in $E_i\cong \pr^{n-1}$.
Given a subset $I\subset \{1, \cdots, n+3\}$, with $|I|\leq n-1$, and an  integer $0\leq s\leq \frac{n-1 -|I|}{2}$, we 
denote by  $J^i_{I,s}$ the join $\Jo\big(\langle q_j\rangle_{j\in I}, \Sec_{s-1}(C')\big)\subset E_i$. 
One can check that 
$$
		E_i\cap J_{I,s} \ = \  \left\{ 
		\begin{aligned}
		& J^i_{I\smallsetminus \{i\},s} \ & \text{ if } i\in I, \\
		& \emptyset \ & \text{ if } i\not\in I \ \text { and } s=0, \\
		& J^i_{I\cup\{i\},s-1} \ & \text{ if } i\not\in I \ \text { and } s\geq 1.
		\end{aligned}
		\right.
$$
Therefore, the strict transform of $E_i$ under $\varphi_1$ is the blow-up of $\pr^{n-1}$ at the points $q_1, \dotsc, q_{n+3}$.
For $2\leq j\leq m-1$, the restriction of $\varphi_j$ to the  strict transform of $E_i$ in $X_{j-1}$ flips  the strict transforms of all $J^i_{I,s}$'s
of dimension $j-1$.
\end{parg}
\begin{parg}\label{dim4}
When $n=4$, the birational map $\varphi_1\colon X=X_0\dasharrow X_1=X^4_{\text{\em Fano}}$ flips $J_{\{ij\},0}$ (strict transform of the line $\overline{p_ip_j}\subset\pr^4$)  for  $1\leq i,j\leq 7$, and $J_{\emptyset,1}$ (strict transform of $C\subset\pr^4$); this yields $22$ among the $64$ special $\pr^2$'s in $X^4_{\text{\em Fano}}$, corresponding to the $64$ extremal rays of $\NE(X^4_{\text{\em Fano}})$. The remaining ones are the strict transforms of the $7$ surfaces $\Jo(\langle p_i\rangle,C)$ and of the $35$ planes 
$\langle p_i,p_j,p_h\rangle$ in $\pr^4$.

Notice in particular that $E_i\subset X$ does not contain any special subvariety $J_{I,s}$, while the strict transform of $E_i$ in $X^4_{\text{\em Fano}}$ contains $7$ special 
 $\pr^2$'s, namely the flipped loci of the flips of  $J_{\{ij\},0}$ for $j\neq i$, and of 
$J_{\emptyset,1}$.
\end{parg}


\section{Pseudo-isomorphisms between $G$ and $X$}\label{relation}
\noindent Let $m$ be a positive integer, and set $n=2m$.
Fix $n+3$ distinct points $(\lambda_1:1),\dotsc,(\lambda_{n+3}:1)\in \pr^1$, and
let $p_1, \dotsc, p_{n+3}\in \pr^n$ be their images under a Veronese embedding $\pr^1\hookrightarrow\pr^n$. 
Let $Z$, $G$ and $X$ be the varieties introduced in Sections~\ref{prel}, \ref{Fano} and \ref{blowup}.
We follow the notation introduced in those sections. 
In this section we determine the nef cone of $G$, and then we prove Theorem~\ref{main}, which follows from  Theorem~\ref{SQM}
and Corollary~\ref{blowupmodel}.
Our aim is to  identify the line bundles on $G$ whose linear systems define rational maps $G\dasharrow \pr^n$ 
inducing a pseudo-isomorphism $G\dasharrow X$. 
This is achieved by combining  the description of  $\Nef(G)\subset H^2(G,\R)$ given by Theorem \ref{nefcone},
and the description of $\Nef(X^n_{\text{\it Fano}})\subset H^2(X,\R)$ in terms of the basis $\{-K_X, E_1, \dots, E_{n+3}\}$ for  $H^2(X,\R)$,
which was obtained from the  Mori chamber decomposition of $\Eff(X)$ in Section~\ref{blowup}.

We first  describe the cones $\Nef(G)$ and  $\NE(G)$. 
For  $n=4$, this was proved in \cite[Theorem 4.3]{borcea91}.
\begin{thm}\label{nefcone} \label{NE(G)}
Let the notation be as above. Then 
$$
 \NE(G)=\Cone(\ell_M)_{M\in\ma{F}_m(G)}=\alpha^{-1}(\ma{E}) \ \text{ and }  \ \Nef(G)=\beta(\ma{E}^{\vee}). 
$$
\end{thm}
\begin{proof}
By Proposition \ref{ell_M}, $\ell_M$ generates an extremal ray of $\NE(G)$ for every $M\in\ma{F}_m(G)$. 
This yields $2^{n+2}$ distinct extremal rays of $\NE(G)$.
On the other hand,  $G\cong X_{\text{\em Fano}}$ by Remark \ref{G=X_Fano}, and 
$\NE(X_{\text{\em Fano}})$ has precisely $2^{n+2}$ extremal rays, as explained in Paragraph \ref{X_Fano}.
So we have: 
$$
\NE(G)=\Cone(\ell_M)_{M\in\ma{F}_m(G)}=\alpha^{-1}(\ma{E}).
$$
The equality $\Nef(G)=\beta(\ma{E}^{\vee})$ follows from the duality between $\Nef(G)$ and $\NE(G)$, and from  Proposition \ref{iso}.
\end{proof}

Similarly, we will show in Proposition~\ref{effdivGMov1G} that 
$\Eff(G)=\beta(\ma{E})$ and $\Mov_1(G)=\alpha^{-1}(\ma{E}^{\vee})$.
So the cones  $\NE(G)$ and $\Eff(G)$ are isomorphic under $\beta\circ\alpha$, and the same holds for $\Mov_1(G)$ and $\Nef(G)$.

\medskip

Recall from Section \ref{Fano} that $E_M=\beta(M)\in H^2(G,\Z)$ for every $M\in\ma{F}_{m}(Z)$.
For each $M\in\ma{F}_{m}(Z)$, consider the linear map
$$
h_{M}\colon H^2(X,\R)\la H^2(G,\R)
$$
defined by:
$$
h_{M}(-K_X)=-K_G\quad\text{and}\quad h_M(E_i)=E_{\sigma_i(M)}\text{ for every }i=1,\dotsc,n+3.
$$
One can check that $h_M$ respects the integral points, namely that it is induced by an isomorphism $H^2(X,\Z)\to H^2(G,\Z)$, 
and that $h_{\sigma_I(M)}=\sigma_I\circ h_M$ for every $I\subseteq\{1,\dotsc,n+3\}$.

We also set 
 \stepcounter{thm}
 \begin{equation}\label{def_h_M_tilde}
\tilde{h}_M:=\beta^{-1}\circ h_M\colon H^2(X,\R)\la H^n(Z,\R),
 \end{equation}
so that $\tilde{h}_{M}(-K_X)=\eta$ and $\tilde{h}_M(E_i)=\sigma_i(M)$ for every $i=1,\dotsc,n+3$.
\begin{lemma}\label{h_M}
For every $M\in\ma{F}_m(Z)$ and $I\subseteq\{1,\dotsc,n+3\}$ of even cardinality, we have:
$$
h_M(E_I)=E_{\sigma_I(M)}, \ h_M(\Eff(X))=\beta(\ma{E}),\ \text{ and }\ h_M\big(\Nef(X_{\text{\it Fano}})\big)=\Nef(G).
$$ 
\end{lemma}
\begin{proof}
Let $I\subseteq\{1,\dotsc,n+3\}$ be such that $|I|=n-2s$ is even, $s\geq 0$.
We can rewrite \eqref{E_I} as 
$$
E_I \ = \ \frac{1}{n+1} \Big( (s+1)(-K_X) - 2(s+1)\sum_{i\in I}E_i +(n-1-2s) \sum_{j\in I^{^c}}E_j\Big).
$$
It follows from \eqref{M_I2} that $\tilde{h}_M(E_I)=\sigma_I(M)$, and hence $h_M(E_I)=E_{\sigma_I(M)}$. 
This implies that $h_M(\Eff(X))=\beta(\ma{E})$.

By comparing \eqref{fanochamber} and \eqref{ineqG}, we see that  $\tilde{h}_M\big(\Nef(X^n_{\text{\em Fano}})\big)=\ma{E}^{\vee}$.
Hence, $h_M\big(\Nef(X^n_{\text{\em Fano}})\big)=\beta(\ma{E}^{\vee})=\Nef(G)$ by Theorem~\ref{nefcone}.
\end{proof}
\begin{proposition}\label{pullback}
Let $\xi\colon G\dasharrow X$ be a pseudo-isomorphism, and consider the induced linear map
$$
\xi^*\colon H^2(X,\R)\la H^2(G,\R).
$$
Then, up to  a unique permutation of $E_1,\dotsc,E_{n+3}\subset X$, there exists a unique $M\in\ma{F}_{m}(Z)$ such that  $\xi^*=h_M$. 
\end{proposition}
\begin{proof}
We have $\xi^*(-K_X)=-K_G$, and hence $\xi^*\big(\Nef(X^n_{\text{\em Fano}})\big)=\Nef(G)$.

We fix $M_{0}\in \ma{F}_{m}(Z)$ and follow the notation introduced in \ref{M_0}.
Consider $\xi^*\circ (h_{M_0})^{-1}\colon H^2(G,\R)\to H^2(G,\R)$. 
By Lemma~\ref{h_M}, this map fixes $-K_G$ and sends $\Nef(G)$ to itself. 
Using the isomorphism $\beta\colon H^n(Z,\R)\to H^2(G,\R)$ and Theorem~\ref{nefcone}, 
we obtain a linear map $f\colon H^n(Z,\R)\to H^n(Z,\R)$ such that $f(\eta)=\eta$ and $f(\ma{E}^{\vee})=\ma{E}^{\vee}$: 
$$
\xymatrix{&{H^2(X,\R)}\ar[dl]_{h_{M_0}}\ar[dr]^{\xi^*}&\\
{H^2(G,\R)}\ar[rr]^{\xi^*\circ (h_{M_0})^{-1}}&&{H^2(G,\R)}\\
{H^n(Z,\R)}\ar[u]^{\beta}\ar[rr]^f&&{H^n(Z,\R).}\ar[u]_{\beta}
}
$$
By Lemma~\ref{linear}, we have $f\in W(D_{n+3})$.

Consider the stabiliser $G_0\subset W(D_{n+3})$ of $M_0$, and recall that $W(D_{n+3})=W'\rtimes G_0$ and $G_0\cong S_{n+3}$. 
Thus there are uniquely defined $\omega\in G_0$, $\sigma_I\in W'$ and $\kappa\in S_{n+3}$ such that 
 $f=\sigma_I\circ\omega$ and $\omega(M_i)=M_{\kappa(i)}$ for every $i=1,\dotsc,n+3$.
Since $\beta$ is $W'$-equivariant, this means that $$\xi^*(E_i)=\beta(f(M_i))=\beta(\sigma_{\kappa(i)}(M_I))=\sigma_{\kappa(i)}
(\beta(M_I))=\sigma_{\kappa(i)}(E_{M_I})$$ for every $i=1,\dotsc,n+3$.
Apply the permutation $\kappa^{-1}$ to $E_1,\dotsc,E_{n+3}\subset X$. 
After this reordering, we get $f=\sigma_I\in W'$ and $\xi^*=\sigma_I\circ h_{M_0} =h_{M_I}$. 
\end{proof}

From now on we order the divisors $E_1,\dotsc,E_{n+3}\subset X$, and correspondingly the points $p_1,\dotsc,p_{n+3}\in\pr^n$, as in Proposition \ref{pullback}. 
At this point we can determine the cone of effective divisors and the cone of moving curves of $G$.
\begin{proposition}\label{effdivGMov1G} 
For every $M\in\ma{F}_m(Z)$, there is a unique effective divisor in $G$ with class $E_M\in H^2(G,\Z)$.
This is a fixed prime divisor which we still denote by $E_M\subset G$.
We have:
$$
\Eff(G)=\beta(\ma{E})=\Cone( E_M)_{M\in\ma{F}_{m}(Z)} \ \text{ and } \ \Mov_1(G)=\alpha^{-1}(\ma{E}^{\vee}).
$$
\end{proposition}
\begin{proof}
By Theorem~\ref{pseudoisom}, there exists a pseudo-isomorphism $\xi\colon G\dasharrow X$. 
By Proposition~\ref{pullback}, there exists $M\in\ma{F}_{m}(Z)$ such that  $\xi^*=h_M$. In particular,  for every $I\subset\{1,\dotsc,n+3\}$ with $|I|$ even, we have $\xi^*(E_I)=E_{\sigma_I(M)}$ by Lemma~\ref{h_M}.
Thus the strict transform in $G$ of $E_I\subset X$ is a fixed prime divisor, and it is the unique effective divisor with class $E_{\sigma_I(M)}$.
It also follows from Lemma~\ref{h_M} that 
$$
\Eff(G)=\xi^*\Eff(X)=\beta(\ma{E})=\Cone (E_M)_{M\in\ma{F}_{m}(Z)}.
$$ 

The equality $\Mov_1(G)=\alpha^{-1}(\ma{E}^{\vee})$ follows from the duality $\Mov_1(G)=\Eff(G)^{\vee}$ and from Proposition \ref{iso}.
\end{proof}

For each $M\in\ma{F}_{m}(Z)$, we set
\stepcounter{thm}
\begin{equation}\label{H_M}
\begin{split}
H_M \ :=& \ h_M(H)\ = \ \frac{1}{n+1}\left(-K_G+(n-1)\sum_{i=1}^{n+3}E_{\sigma_i(M)}\right) \\
       =& \ m(-K_G)-(n-1)E_M \ \in \ H^2(G,\Z),
\end{split}
\end{equation}
where the last equality follows from \eqref{M_I2} (taking $M=M_0$ and $I=\emptyset$), 
using the isomorphism $\beta\colon H^n(Z,\R)\to H^2(G,\R)$. 
\begin{thm}\label{SQM}
For every $M\in\ma{F}_{m}(Z)$, the divisor class $H_M$ is movable, and its complete linear system 
defines a birational map 
$$
\rho_M\colon G\dasharrow\pr^n,
$$ 
with exceptional divisors $E_{\sigma_1(M)},\dotsc,E_{\sigma_{n+3}(M)}$,
inducing a pseudo-isomorphism 
$$
\xi_M\colon G\dasharrow X
$$ 
whose induced map $\xi_M^*\colon H^2(X,\R)\to H^2(G,\R)$ coincides with $h_M$.

For every $I\subseteq\{1,\dotsc,n+3\}$, $\rho_{\sigma_I(M)}=\rho_M\circ\sigma_I$ and $\xi_{\sigma_I(M)}=\xi_M\circ\sigma_I$.
\end{thm}
\begin{proof}
By Theorem~\ref{pseudoisom}, there exists a pseudo-isomorphism $\xi\colon G\dasharrow X$. 
Let $\rho\colon G\dasharrow \pr^n$ be the composition of $\xi$ with the blow-up morphism $X\to\pr^n$. 
 
By Proposition~\ref{pullback}, 
there exists $M_0\in\ma{F}_m(Z)$ such that $\xi^*=h_{M_0}$. 
This implies that $\rho^*(\ma{O}_{\pr^n}(1))=H_{M_0}$. 
Hence the class $H_{M_0}$ is movable, and $H^0(G,H_{M_0})\cong H^0(\pr^n,\ma{O}_{\pr^n}(1))$.
This proves the first statement for $M=M_0$, with $\rho_{M_0}=\rho$ and $\xi_{M_0}=\xi$.

Let $I\subseteq\{1,\dotsc,n+3\}$. We use the notation introduced in \ref{M_0}.
The automorphism $\sigma_I\colon G\to G$ fixes $-K_G$ and maps $E_{M_0}$ to $E_{M_I}$, hence it maps $H_{M_0}$ to $H_{M_I}$. 
This yields the first statement for $M=M_I$, with $\rho_{M_I}=\rho\circ\sigma_I$ and $\xi_{M_I}=\xi\circ\sigma_I$.

The last statement is clear.
\end{proof}
\begin{corollary}\label{blowupmodel}
Let $\w{X}$ be any blow-up of $\pr^n$ at $n+3$ points. If $\w{X}$ is pseudo-isomorphic to $G$, 
then $\w{X}$ is isomorphic to $X$. 
\end{corollary}
\begin{proof}
Let  $\w{\xi}\colon G\dasharrow \w{X}$ be a pseudo-isomorphism, 
and let $\w{\rho}\colon G\dasharrow \pr^n$ be the composition of $\w{\xi}$ with the blow-up morphism $\w{X}\to\pr^n$. 
Then $\w{\rho}$ has $n+3$ exceptional prime divisors, whose classes must generate a simplicial facet of $\Eff(G)$. 
By Proposition~\ref{effdivGMov1G} and the description of the facets of $\ma{E}$ in Remark~\ref{rem:facets}, 
every simplicial facet of $\Eff(G)$ is generated by $E_{\sigma_1(M)},\dotsc,E_{\sigma_{n+3}(M)}$ for some $M\in\ma{F}_m(Z)$. 
Since each $E_{\sigma_i(M)}$ is unique in its linear system, 
$\w{\rho}\colon G\dasharrow\pr^n$ and $\rho_{M}\colon G\dasharrow \pr^n$ have the same exceptional divisors. 
This means that $\w{\rho}$ and $\rho_{M}$ coincide up to a projective transformation of $\pr^n$, and therefore $\w{X}\cong X$.
\end{proof}
\begin{remark}[Comparing the intersection product in $H^n(Z,\Z)$ with  Dolgachev's pairing on $H^2(X,\Z)$]
In \cite{dolgachev}, Dolgachev defined a non-degenerate symmetric bilinear form $(\ , \  )$ on $H^2(X,\Z)$,  
by imposing that the basis $H,E_1,\dotsc,E_{n+3}$ is orthogonal, 
$$(H,H)=n-1\quad\text{and}\quad (E_i,E_i)=-1\text{ for all }i=1,\dotsc,n+3.$$
This pairing has signature $(1,n+3)$, and $(-K_X,-K_X)=4(n-1)$.
Consider $\tilde\eps_i\in H^2(X,\R)$ defined in \eqref{epstilde}:
$$\tilde\eps_i:=\frac{1}{2}\Big(H-\sum_{j\neq i}E_j+E_i\Big)\quad\text{for }
i=1,\dotsc,n+3.$$
Then we have
$$
(-K_X,\tilde{\eps}_i)=0\ \text{ and }\ (\tilde{\eps}_i,\tilde{\eps}_j)=-\delta_{ij}\quad \text{for every }i,j=1,\dotsc,n+3,
$$
thus $-K_X,\tilde\eps_1,\dotsc,\tilde\eps_{n+3}$ is another orthogonal basis for $H^2(X,\R)$.

Fix $M_0\in\ma{F}_m(Z)$, and consider the orthogonal basis $\eta,\eps_1,\dotsc,\eps_{n+3}$ for $H^n(Z,\R)$ introduced in \eqref{eps}. 
Recall that $\eta^2=4$ and $\eps_i^2=(-1)^m$ for every $i=1,\dotsc,n+3$.
Consider the isomorphism introduced in \eqref{def_h_M_tilde}:
$$
\tilde{h}_{M_0}\colon H^2(X,\R)\to H^n(Z,\R).
$$
From \eqref{H_M} and \eqref{eps_i} we have $\tilde{h}_{M_0}(\tilde{\eps}_i)=\eps_i$ for every $i=1,\dotsc,n+3$.
Therefore $\tilde{h}_{M_0}$ maps an orthogonal basis for Dolgachev's pairing in $H^2(X,\R)$, to an
orthogonal basis for the intersection product in $H^n(Z,\R)$.
In particular $\tilde{h}_{M_0}$ sends the $D_{n+3}$-lattice $(-K_X)^{\perp}\subset H^2(X,\Z)$ to the $D_{n+3}$-lattice $\eta^{\perp}\subset H^n(Z,\Z)$, 
and the restriction of  $\tilde{h}_{M_0}$ to  these lattices is an isometry up to the sign $(-1)^{m-1}$.
(Notice that $\tilde{h}_{M_0}$ is globally an isometry if and only if $n=2$.) 
This also shows that $\tilde{h}_{M_0}$ is $W(D_{n+3})$-equivariant.
\end{remark}


\section{Cones of curves and divisors in $G$}\label{section:conesG}
\noindent Let the setup be as in Section~\ref{relation}.
Recall that in Section~\ref{blowup} we considered the cones
$$
\Nef(X^n_{\text{\em Fano}})  \ \subset \ \Mov^1(X) \ \subset \ \Eff(X) \ \subset \  H^2(X,\R),
$$
the affine hyperplane $\ma{H}\subset H^2(X,\R)$ containing all the $E_I$'s, and the polytopes given by the intersections of these cones with $\ma{H}$:
$$
\Delta_{\text{\em Fano}} \ \subset \ \Delta_{\text{\em Mov}} \ \subset \ \Delta \ \subset \  \ma{H} \ \cong \ \R^{n+3}.
$$
From the  linear inequalities defining these polytopes  in $\R^{n+3}$, and the expression  \eqref{eq:phi} of the radial projection onto $\ma{H}$,
one can write explicitly the linear inequalities defining the cones $\Nef(X^n_{\text{\em Fano}})\cong\Nef(G)$, $\Mov^1(X)\cong\Mov^1(G)$, and 
$\Eff(X)\cong\Eff(G)$  with respect to the basis $H,E_1,\dotsc,E_{n+3}$ of $H^2(X,\R)$. 
Inequalities defining $\Mov^1(X)$ and $\Eff(X)$ were obtained in a different way in  \cite{BDP}.
In this section, we  reinterpret the facets and extremal rays of these cones in terms of special divisors and curves  in $G$.

Recall from Section~\ref{prel}  that $\ma{E}\subset H^n(Z,\R)$ is the cone over the demihypercube $\Delta$ with vertices $\{M\}_{M\in\ma{F}_{m}(Z)}$.
Its dual cone $\ma{E}^{\vee}\subset \ma{E}$ has $2(n+3)+2^{n+2}$ extremal rays, generated by the classes:
$$
\big\{M+{\sigma_i(M)} \ \big| \ M\in\ma{F}_m(Z), i\in \{1, \dotsc, n+3\}\big\} \ \cup \ 
$$
$$
\left\{  \delta_M =  \left\lfloor\frac{m+1}{2}\right\rfloor\eta+(-1)^{m}M \right\}_{M\in\ma{F}_{m}(Z)}.
$$
For a fixed  $i\in\{1,\dotsc,n+3\}$,  there are two distinct classes  $M+\sigma_i(M)$ as $M$ varies in $\ma{F}_m(Z)$,
and they form an orbit for the action of $W'$ on $H^n(Z,\Z)$. 
The stabilizer of this orbit is the subgroup 
$G_i:=\{\sigma_I\,|\,i\not\in I\text{ and }|I|\text{ is even}\}$. 
The group  $W'$ acts transitively and freely on the set $\big\{\delta_{M}\big\}_{M\in\ma{F}_{m}(Z)}$.
The facet of $\ma{E}$ corresponding to each extremal ray of $\ma{E}^{\vee}$ was described in  Remark~\ref{rem:facets}:
\begin{enumerate}[--]
\item $\big(M+\sigma_i(M)\big)^{\perp}\cap\ma{E}$ is the cone over the $(n+2)$-dimensional demihypercube with vertices 
	    $\big\{\sigma_I(M)\  \big| \ I\subset \{1, \dotsc, n+3\}\smallsetminus \{i\}, |I|\not\equiv m \text{ mod } 2 \}$.
\item $(\delta_M)^{\perp}\cap\ma{E}$ is a simplicial cone generated by the classes $\sigma_i(M)$,  $i\in\{1,\dotsc,n+3\}$. 
\end{enumerate}

Now we turn to cones of curves and divisors in $G$. 
We showed in Theorem~\ref{NE(G)} and Proposition~\ref{effdivGMov1G} that 
$$
\Nef(G) \ = \ \beta(\ma{E}^{\vee})\ \subset \ \beta(\ma{E})\ = \ \Eff(G), \text { and }
$$
$$
\Mov_1(G) = \ \alpha^{-1}(\ma{E}^{\vee})\ \subset \  \alpha^{-1}(\ma{E})\ = \ \NE(G). 
$$
We give a geometric description of the facets and extremal rays of these cones in terms of special divisors and curves  in $G$.
\begin{parg}[{$\Eff(G)$}]\label{parg:Eff(G)}
The cone $\Eff(G)$ has $2^{n+2}$ extremal rays, generated by the classes $\{ E_M\}_{M\in\ma{F}_{m}(Z)}$.
Each $E_M$ is a fixed prime divisor. 
The group  $W'\subset \Aut(G)$ acts transitively and freely on the set  $\{ E_M\}_{M\in\ma{F}_{m}(Z)}$.
In particular, all these divisors are isomorphic, and they can be described as a small modification of 
the blow-up of $\pr^{n-1}$ at $n+3$ points contained in a rational normal curve 
(see Paragraph~\ref{X-->X_Fano} for a precise description).  
\end{parg}
\begin{parg}[The divisor $E_M$ when $n=4$]
 Set $n=4$; in this case $E_M$ is isomorphic to the blow-up  of $\pr^{3}$ at $7$ points contained in a rational normal curve. To describe geometrically $E_M$ inside $G$, consider 
  the closed subset
$$\big\{[L]\in G\,|\,L\cap M\neq\emptyset\big\}.$$ 
Then this locus is not equidimensional, and
$E_M$ is its unique divisorial component.

Indeed, let us consider again the incidence diagram 
$$
\xymatrix{
&{\mathcal{I}}\ar[dl]_{\pi}\ar[dr]^e&\\
G&&Z
}
$$
as in \ref{alphabeta}, so that $\dim\mathcal{I}=5$, $\pi$ is a $\pr^1$-bundle, and
$\{[L]\in G\,|\,L\cap M\neq\emptyset\}=\pi(e^{-1}(M))$.
For the purposes of this paragraph only, it is better to denote by $[M]\in H^4(Z,\Z)$ the fundamental class of the plane $M\subset Z$.

 It is not difficult to see that $e$ is flat, so that  $e^{-1}(M)$ is equidimensional of dimension $3$, and $e^*([M])=[e^{-1}(M)]\in H^4(\mathcal{I},\Z)$. 
Then $\beta([M])=\pi_*e^*([M])=[\pi_*(e^{-1}(M))]$. By Proposition \ref{effdivGMov1G}, we have $E_M=\pi_*(e^{-1}(M))$, so that $E_M$ is the unique divisorial component of $\pi(e^{-1}(M))$.

Now let us consider the planes $M^*,\sigma_1(M)^*,\dotsc,\sigma_7(M)^*\subset G$ (see \eqref{M*}); they are all contained in  $\pi(e^{-1}(M))$.

Let $i\in\{1,\dotsc,7\}$. Recall that $\ell_{\sigma_i(M)}\subset \sigma_i(M)^*$ is a line, and that $\ell_{\sigma_i(M)}=\alpha(\sigma_i(M))$. By Proposition \ref{iso}, using for instance \eqref{M_I}, 
we have 
$$E_M\cdot\ell_{\sigma_i(M)}=M\cdot\sigma_i(M)=-1,$$ 
so that  $\sigma_i(M)^*\subset E_M$. On the other hand $E_M$ contains only $7$ planes $(M')^*$ (see \ref{dim4}), therefore $M^*$ cannot be contained in $E_M$.
This shows that $M^*$ is a $2$-dimensional irreducible component of  $\pi(e^{-1}(M))$.
\end{parg}
\begin{parg}[{$\NE(G)$}]\label{parg:NE(G)}
The cone $\NE(G)$ has $2^{n+2}$ extremal rays, generated by the classes  $\{ \ell_M\}_{M\in\ma{F}_{m}(Z)}$,
on which $W'\subset \Aut(G)$ acts transitively. 
The contraction of the extremal ray generated by  $\ell_M$ contracts  $M^*\cong\pr^m$  to a point.

Fix $M\in\ma{F}_{m}(Z)$ and consider the pseudo-isomorphism  $\xi_M\colon G\dasharrow X$ from Theorem~\ref{SQM}.  
This fixes an identification of $G$ with $X^n_{\text{\em Fano}}$ which identifies each divisor $E_{\sigma_I(M)}\subset G$ with the strict transform 
of the divisor $E_I\subset X$.
Let $I\subset \{1, \dotsc, n+3\}$ be such that $|I|\leq m+1$.
It follows from the discussion in Paragraph~\ref{X-->X_Fano} that 
\begin{enumerate}[--]
\item If $|I|\not\equiv m \text{ mod } 2 $, then $\big(\sigma_I(M)\big)^*\subset G$ is the strict transform of $J_{I,s}\subset X$,
	    where $s=\frac{m+1-|I|}{2}$. 
\item If $|I|\equiv m \text{ mod } 2 $, then $\big(\sigma_I(M)\big)^*\subset G$ is the flipped locus of the flip of the 
	    strict transform of $J_{I,s}\subset X$, where $s=\frac{m-|I|}{2}$. 
\end{enumerate}
In particular, we see that $(M')^*\subset E_M$ if and only if $M'=\sigma_I(M)$ for some $I\subset \{1, \dotsc, n+3\}$ 
with $|I|\leq m-1$ and $|I|\not\equiv m \text{ mod } 2 $.
\end{parg}
\begin{parg}[{$\Nef(G)$}]\label{parg:Nef(G)}
The cone $\Nef(G)$ has  $2^{n+2}+2(n+3)$ extremal rays, generated by the 
classes 
$$
\big\{D_M=\beta(\delta_M)\big\}_{M\in\ma{F}_{m}(Z)} \ \cup 
\big\{E_M+E_{\sigma_i(M)}\,\big|\,M\in\ma{F}_m(Z), \ i=1,\dotsc,n+3\big\}.
$$
For fixed $i$, the morphisms associated to the extremal rays generated by $E_M+E_{\sigma_i(M)}$ and $E_{\sigma_j(M)}+E_{\sigma_{ij}(M)}$ ($j\neq i$) are
 the generic 
$\pr^1$-bundles $\varphi_i\colon G\to Y_{\varphi_i}$ and $\psi_i\colon G\to Y_{\psi_i}$ described in Lemma~\ref{fibertype}.
The morphism associated to the extremal ray generated by $D_M$ is the composition of the (disjoint) small
contractions of $\sigma_i(M)^*\subset G$ to a point, $i=1,\dotsc,n+3$. 
\end{parg}
\begin{parg}[{$\Mov_1(G)$}]\label{parg:Mov_1(G)}\label{Y_ph_i}
The cone $\Mov_1(G)$ has  $2(n+3)+2^{n+2}$ extremal rays, generated by the curve classes:
$$
\left\{\ell_M+\ell_{\sigma_i(M)}\,|\,M\in\ma{F}_{m}(Z),i=1,\dotsc,n+3\right\}\cup\left\{d_M\,|\,M\in \ma{F}_{m}(Z)\right\},
$$
where 
$$
d_M:=\alpha^{-1}(\delta_M)=\left\lfloor\frac{m+1}{2}\right\rfloor\alpha^{-1}(\eta)+(-1)^m\ell_M\in\N(G).
$$

For a fixed  $i\in\{1,\dotsc,n+3\}$,  there are two distinct classes  $\ell_M+\ell_{\sigma_i(M)}$ as $M$ varies in $\ma{F}_m(Z)$,
and they form an orbit for the action of $W'$ on $\N(G)$.
By Corollary~\ref{extremalfaces}, these are the classes of the fibers of the generic $\pr^1$-bundles $\varphi_i\colon G\to Y_{\varphi_i}$
and $\psi_i\colon G\to Y_{\psi_i}$.
Under the identification $G\cong X^n_{\text{\em Fano}}$ induced by a pseudo-isomorphism $G\dasharrow X$, 
these correspond to the the generic $\pr^1$-bundles $\phi_i, \phi_i'\colon  X^{n}_{\text{\em Fano}}\to (X_{\ma{P}_i})^{n-1}_{\text{\em Fano}}$
described in Lemma~\ref{image_phi}.
In particular, we see that $Y_{\varphi_i}\cong Y_{\psi_i}\cong (X_{\ma{P}_i})^{n-1}_{\text{\em Fano}}$.

As for the class $d_M$, using Proposition~\ref{iso} and Remark~\ref{rem:facets}, one computes:
$$
-K_G\cdot d_M=\eta\cdot \delta_M=n+1, \ \text{ and }
$$
 $$
 E_{\sigma_i(M)} \cdot d_M  =\sigma_i(M)\cdot\delta_M=0\text{ for every }i=1,\dotsc,n+3.
 $$
 Therefore $d_M$ is the class of the strict transform in $G$ of a general line in $\pr^n$ under the map $\rho_M\colon G\dasharrow \pr^n$. 
\end{parg}

In order to complete the picture, next we describe equations for the movable cone $\Mov^1(G)  \subset  H^2(G,\R)$ 
and give a geometric description of the extremal rays of the dual cone $\Mov^1(G)^{\vee}  \subset  \N(G)$.
We do this for $n\geq 4$, since when $n=2$ we have $\Mov^1(G)=\Nef(G)$ and $\Mov^1(G)^{\vee}=\NE(G)$.
\begin{proposition}\label{movG}
Suppose that $n\geq 4$. The cone $\Mov^1(G)^{\vee}\subset\N(G)$ has $2^{n+2}+2(n+3)$ extremal rays, generated by the classes 
$$
\left\{e_M\,|\,M\in\ma{F}_m(Z)\right\}\cup \left\{\ell_M+\ell_{\sigma_i(M)}\,|\,M\in\ma{F}_m(Z),i=1,\dotsc,n+3\right\},
$$
where $e_M:=\left\lfloor\frac{m}{2}\right\rfloor\alpha^{-1}(\eta)+(-1)^{m-1}\ell_M$.
\end{proposition}
\begin{proof}
Recall from Section~\ref{blowup} that the intersection of  $\Mov^1(X)$ with the affine hyperplane $\ma{H}\subset H^2(X,\R)$
is given by:
$$
\Delta_{\text{\em Mov}} \ =  \  \left\{ 
		\begin{aligned}
		& -\frac{1}{2}\leq \alpha_i \leq \frac{1}{2}, \ & i\in \{1, \dots, n+3\} \\
		&H_I\geq 2 , \ & |I| \text{ odd.}
		\end{aligned}
		\right.
$$
So $\Mov^1(G)=\beta\big(\ma{M}\big)$, where $\ma{M}$ is the cone over $\Delta_{\text{\em Mov}}$, now viewed as a polytope in the hyperplane
$\{\gamma\,|\,\gamma\cdot\eta=1\}\subset H^n(Z,\R)$.

Notice that the facet $(H_I=2)\cap \Delta_{\text{\em Mov}}$ of $\Delta_{\text{\em Mov}}$ is the convex hull of the vertices $v_J$ such that 
$ \# (I\smallsetminus J) + \# (J\smallsetminus  I) =2$. This follows from \eqref{H_I(v_J)}.
In the same way done in Section~\ref{prel} for $\ma{E}$, one can use the linear inequalities defining $\Delta_{\text{\em Mov}}$ to compute the linear
inequalities defining $\ma{M}$, or equivalently the generators of the dual cone $\ma{M}^{\vee}$. These are:
$$
\big\{M+{\sigma_i(M)} \ \big| \ M\in\ma{F}_m(Z), i\in \{1, \dotsc, n+3\}\big\} \ \cup \ 
\left\{  \eta_M  \right\}_{M\in\ma{F}_{m}(Z)},
$$
where $\eta_M=\left\lfloor\frac{m}{2}\right\rfloor \eta+(-1)^{m-1}M$ (notice that $e_M=\alpha(\eta_M)$).
Indeed, one can check using \eqref{M_I} that
 \stepcounter{thm}
\begin{equation} \label{eta_M}
\eta_M \cdot \sigma_{ij}(M) \ = \ 0 \ \ \forall i\neq j. 
\end{equation}
By the duality properties of $\alpha$ and $\beta$, we have $\Mov^1(G)^{\vee}=\alpha^{-1}(\ma{M}^{\vee})$,
and the result follows. 
\end{proof}
\begin{parg}\label{parg:Mov^1(G)}
The classes $\ell_M+\ell_{\sigma_i(M)}$ were described in Paragraph~\ref{Y_ph_i} above.
Now we want to describe the classes $e_{M}$.

Given $M\in\ma{F}_{m}(Z)$ and $i\in\{1, \ldots, n+3\}$, set $M_0=\sigma_i(M)$, and follow the notation introduced in 
\ref{M_0}, so that $M=M_i$.
Consider the pseudo-isomorphism  $\xi_{M_0}\colon G\dasharrow X$ from Theorem~\ref{SQM},  
and note that the divisor $E_M\subset G$ is the strict transform  of the divisor $E_i\subset X$ under $\xi_{M_0}$.
By \eqref{eta_M} above, we have that 
$$
E_{M_j}\cdot e_M \ = \ 0 \ \ \forall j\neq i. 
$$
Similarly one computes that $E_{M}\cdot e_M=-1$. 
We conclude that $e_{M}$ is the class of the strict transform under $\xi_{M_0}^{-1}$ of a general line in $E_i\cong\pr^{n-1}$.
\end{parg}
\begin{remark}
Set $c:=\alpha^{-1}(\eta)\in\N(G)$. We have:
$$
-K_G\cdot c=4\quad\text{and}\quad E_M\cdot c=1\text{ for every $M\in\ma{F}_m(Z)$.}
$$ 
The class $c$ is fixed by the action of $W(D_{n+3})$ and sits in the interior of the cone $\Mov_1(G)\subset\NE(G)$. 
Let  $M\in\ma{F}_m(Z)$ and consider the rational map $\rho_{M}\colon G\dasharrow \pr^n$ from Theorem~\ref{SQM}. 
Then $c$ is the class of the strict transform via $\rho_{M}^{-1}$ of an elliptic curve of degree $n+1$  in $\pr^n$ through $p_1,\dotsc,p_{n+3}$. 
There is a $4$-dimensional family of such curves (see \cite{dolgannali}). 
\end{remark}
\begin{remark}
In \cite{BDP}, the effective cone $\Eff^1(X)\subset H^2(X,\R)$ is described by 3 sets 
of linear inequalities ($A_n$), ($B_n$) and ($C_{n,t}$).
Similarly, the movable cone $\Mov^1(X)\subset H^2(X,\R)$ is described by 3 sets 
of linear inequalities ($A_n$), ($B_n$) and ($D_{n,t}$) (see \cite[Theorems~5.1 and 5.3]{BDP}).
These are related to the extremal rays of $\Mov_1(G)$ and $\Mov^1(G)^{\vee}$  described in Paragraphs~\ref{parg:Mov_1(G)} 
and \ref{parg:Mov^1(G)} as follows.
A divisor class $D\in H^2(G,\R)$ satisfies the inequalities ($A_n$) and ($B_n$)  if and only if:
  $$
  D\cdot(\ell_M+\ell_{\sigma_i(M)}) \ \geq \ 0 \quad\text{ for every } M\in\ma{F}_m(Z)\text{ and } i=1,\dotsc,n+3.
  $$
It satisfies the inequalities ($C_{n,t}$)  if and only if:
 $$
 D\cdot d_{M} \ \geq 0 \ \quad\text{for every } M\in\ma{F}_m(Z).
 $$
  Finally, it  satisfies the inequalities ($D_{n,t}$) if and only if:
  $$
  D\cdot e_{M}\ \geq 0 \ \quad\text{for every } M\in\ma{F}_m(Z).
  $$
\end{remark}
\begin{parg}[$\MCD(G)$]
Consider the subdivision in polytopes of the demihypercube $\Delta\subset\ma{H}\subset H^n(Z,\R)$ given by the hyperplane arrangement \eqref{eq:MCD_Delta}. By taking the cones over these polytopes, and using  
the isomorphism $\beta\colon H^n(Z,\R)\to H^2(G,\R)$, this subdivision yields the fan $\MCD(G)$.

Fix $M_0\in\ma{F}_m(Z)$ and consider the orthogonal basis $\eps_1,\dotsc,\eps_{n+3}$ of $\eta^{\perp}\subset H^n(Z,\R)$ introduced in \eqref{eps}, 
and the affine coordinates $\alpha_1,\dotsc,\alpha_{n+3}$ in the hyperplane $\ma{H}:=\{\gamma\,|\,\gamma\cdot\eta=1\}$ described on page \pageref{alpha}. 
The group $W'$ fixes $\ma{H}$ and $\eta$, thus it acts linearly in the coordinates $\alpha_i$. More precisely it follows from
 \eqref{standardaction} that, if $I\subset\{1,\dotsc,n+3\}$ has even cardinality, then
$\sigma_I(\alpha_1,\dotsc,\alpha_{n+3})=(\alpha'_1,\dotsc,\alpha'_{n+3})$ with
$$
\alpha_i'=\begin{cases}\alpha_i&\text{if }i\not\in I,\\-\alpha_i&\text{if }i\in I.\end{cases}
$$
The group $W'$ fixes both $\Delta$ and $\Delta_{\text{\em Mov}}$, while  the $2^{n+2}$ polytopes $\sigma_I(\Delta_{\text{\em Nef}})$ are all distinct. The corresponding cones in $\MCD(G)$ are $\xi_{M_I}^*(\Nef(X))=\sigma_I^*(\xi_{M_0}^*(\Nef(X))$.
\end{parg}


\section{The automorphism group of $G$}\label{automorphisms}
\noindent Let the setup be as in Section~\ref{relation}.
In this section we describe the automorphism group of the Fano variety $G$, generalizing
the description of the automorphism group of a quartic del Pezzo surface in Example~\ref{surfaces}.
\begin{proposition}\label{Aut(G)}
There are inclusion of groups:
$$
(\Z/2\Z)^{n+2}\cong W' \subseteq \Aut(G)\subseteq W(D_{n+3})\cong (\Z/2\Z)^{n+2} \rtimes S_{n+3}.
$$

Moreover, if the points  $(\lambda_1:1),\dotsc,(\lambda_{n+3}:1)\in \pr^1$ are general,  then $\Aut(G)=W'\cong(\Z/2\Z)^{n+2}$.
\end{proposition} 
\noindent Notice that in the general case we also have $\Aut(Z)=W'$ \cite[Lemma 3.1]{reidthesis}, so that $Z$ and $G$ have the same automorphism group.
\begin{proof}
Clearly we have $W' \subseteq \Aut(G)$.

For any automorphism $\zeta\in\Aut(G)$, the induced isomorphism  $\zeta^*\colon H^2(G,\R)\to H^2(G,\R)$ preserves $-K_G$ and $\Eff(G)$. 
As in the proof of Proposition~\ref{pullback}, one shows that  $\zeta^*\in W(D_{n+3})$.
This yields a group  homomorphism
 $$
 \Aut(G)\la W(D_{n+3}).
 $$

 Fix $M_0\in\ma{F}_m(Z)$. 
 Consider the stabilizer $G_0$ of $M_0$ in $W(D_{n+3})$, and recall that $W(D_{n+3})=W'\rtimes G_0\cong (\Z/2\Z)^{n+2} \rtimes S_{n+3}$. 
 So, given $\zeta\in\Aut(G)$, there are unique elements $\omega\in G_0$ and $\sigma_I\in W'$ such that $\zeta^*=\omega\circ\sigma_I$.
 Set $\tilde{\zeta}:=\sigma_I\circ\zeta\in\Aut(G)$. Then $\tilde{\zeta}^*=\zeta^*\circ\sigma_I=\omega$, so 
 $\tilde{\zeta}^*$ fixes $E_{M_0}$, and hence it also fixes $H_{M_0}$.
 
Consider the rational map $\rho_{M_0}\colon G\dasharrow\pr^n$ induced by $H_{M_0}$, which contracts the divisors 
$E_{M_1}, \dots, E_{M_{n+3}}$ to the points $p_1, \dots, p_{n+3}$ (see Theorem~\ref{SQM}). 
Then $\tilde{\zeta}^*(\rho_{M_0}^*(\ma{O}_{\pr^n}(1)))=\rho_{M_0}^*(\ma{O}_{\pr^n}(1))=H_{M_0}$, 
so $\rho_{M_0}$ and $\rho_{M_0}\circ\w{\zeta}$ differ by a projective transformation $f\in\Aut(\pr^n)$ preserving the set of points $\{p_1, \dots, p_{n+3}\}$:
$$
\xymatrix{
G\ar[r]^{\tilde{\zeta}}\ar@{-->}[d]_{\rho_{M_0}}&G\ar@{-->}[d]^{\rho_{M_0}}\\
{\pr^n}\ar[r]^{f}&{\pr^n.}
}
$$
In particular, if the points $p_1, \dots, p_{n+3}$ are general, then  $f=\Id_{\pr^n}$,  and  so $\zeta=\sigma_I$. 

Suppose that  $\zeta^*=\Id_{H^2(G,\R)}$.  Then $\tilde{\zeta}=\zeta$ and $f$ must fix each $p_i$.
Since $p_1,\dotsc,p_{n+3}$ are  in  general linear position, this implies that $f=\Id_{\pr^n}$,  and hence $\zeta=\tilde{\zeta}=\Id_G$. 
This shows  that the homomorphism $\Aut(G)\to W(D_{n+3})$ is injective, yielding the statement.
 \end{proof}
Every automorphism of $X$ is induced by a projective transformation of $\pr^n$ preserving the set $\{p_1,\dotsc,p_{n+3}\}$.
This in turns corresponds to a projective transformation of $\pr^1$ preserving the set of points $\{(\lambda_1:1),\dotsc,(\lambda_{n+3}:1)\}\subset\pr^1$.
In particular, if  $\lambda_1,\dotsc,\lambda_{n+3}$ are general, then $\Aut(X)=\{\Id_X\}$.

For any projective variety $Y$, we denote by $\Bir^0(Y)$ the group of \emph{pseudo-automorphisms} of $Y$. 
These are  birational maps $Y\dasharrow Y$ which are isomorphisms in codimension one. 

Since $X$ and $G$ are pseudo-isomorphic, we have $\Bir^0(X)\cong\Bir^0(G)$. 
On the other hand, since $G$ is a Fano manifold, we have $\Bir^0(G)=\Aut(G)$.
Indeed if $\zeta\in\Bir^0(G)$, then $\zeta^*(-K_G)=-K_G$. 
Since $\zeta$ is an isomorphism in codimension one, and $-K_G$ is ample, $\zeta$ must be regular, and similarly for $\zeta^{-1}$. 
\begin{remark}[Explicit description of pseudo-automorphisms of $X$]\label{rem_Dol}
The action of $W'$ on $X$ by pseudo-automorphisms is described by Dolgachev in \cite[\S4.4 - 4.6]{dolgannali}. 
Up to a projective transformation, we may assume that $p_1,\dotsc,p_{n+1}$ are the coordinate points, 
$p_{n+2}=(1:\cdots:1)$, and $p_{n+3}=(a_0:\cdots:a_{n+3})$. 
Since no $n+1$ of the points lie on a hyperplane, all the $a_j$'s are nonzero.

Consider the standard Cremona map centered at $p_1,\dotsc,p_{n+1}$:
$$
s: \ \left(z_0:\cdots:z_n\right)\ \mapsto \ \left(\frac{1}{z_0}:\cdots:\frac{1}{z_n}\right).
$$ 
It is regular at $p_{n+2}$ and $p_{n+3}$, which are mapped respectively to itself and to $(\frac{1}{a_0}:\cdots:\frac{1}{a_n})$. 
The  projective trasformation 
$$
r: \left(z_0:\cdots:z_n\right)\ \mapsto \ \left(a_0z_0:\cdots:a_nz_n\right)
$$ 
fixes $p_1,\dotsc,p_{n+1}$, maps $p_{n+2}$ to $p_{n+3}$, and maps  $(\frac{1}{a_0}:\cdots:\frac{1}{a_n})$ to $p_{n+2}$. 
So the composition 
$$
f_{n+2,n+3}=r\circ s \colon \pr^n\dasharrow\pr^n
$$
induces a pseudo-automorphism $\omega_{n+2,n+3}\colon X\dasharrow X$. 

Similarly, for every $i,j\in\{1,\dotsc,n+3\}$ with $i<j$, we can define a birational involution $f_{ij}\colon\pr^n\dasharrow \pr^n$, 
which is not regular only at $\{p_1,\dotsc,p_{n+3}\}\smallsetminus\{p_i,p_j\}$, and exchanges $p_i$ and $p_j$. 
This induces a pseudo-automorphism $\omega_{ij}\colon X\dasharrow X$.

One can check that $\omega_{ij}^*$ acts on $H^2(X,\Z)$ as follows:
\begin{align*}
\omega_{ij}^*(-K_X)&=-K_X,\quad
\omega_{ij}^*(E_i)=E_j,\quad
\omega_{ij}^*(E_j)=E_i
\\
\omega_{ij}^*(H)&=nH-(n-1)\left(\sum_{h=1}^{n+1}E_h-E_i-E_j\right) \\
\omega_{ij}^*(E_r)&=H-\sum_{h=1}^{n+3}E_h+E_i+E_j+E_r\\
&=\frac{1}{n+1}(-K_X)-\frac{2}{n+1}\sum_{h=1}^{n+3}E_h+E_i+E_j+E_r
\quad \text{ for }r\neq i,j.
\end{align*}

Consider  the isomorphism $\tilde{h}_{M_0}\colon H^2(X,\R)\to H^n(Z,\R)$ defined in \eqref{def_h_M_tilde}, 
and the corresponding action of $\omega_{ij}^*$ on $H^n(Z,\R)$. 
We have:  
$$
\omega_{ij}^*(\eta)=\eta \ \text{ and } \ \omega_{ij}^*(\eps_r)=
\begin{cases}
-\eps_r\text{ if }r=i,j\\
\eps_r\text{ if } r\neq i,j.
\end{cases}
$$
(The latter can be checked using  \eqref{eps_i}.)
Hence  $\omega_{ij}^*=\sigma_{ij}$ and $\omega_{ij}$ is the pseudo-automorphism of $X$ induced by $\sigma_{ij}\in W'$. 
In particular,  the pseudo-automorphism of $X$ induced by $\sigma_1\in W'$ is $\omega_{23}\omega_{45}\cdots\omega_{n+2,n+3}$, and so on.
\end{remark}

\providecommand{\noop}[1]{}
\providecommand{\bysame}{\leavevmode\hbox to3em{\hrulefill}\thinspace}
\providecommand{\MR}{\relax\ifhmode\unskip\space\fi MR }
\providecommand{\MRhref}[2]{%
  \href{http://www.ams.org/mathscinet-getitem?mr=#1}{#2}
}
\providecommand{\href}[2]{#2}


\begin{thebibliography}{BHK10}

\bibitem[AM72]{artinmumford}
M.~Artin and D.~Mumford, \emph{Some elementary examples of unirational
  varieties which are not rational}, Proc.\ London Math.\ Soc.\ \textbf{25}
  (1972), 75--95.

\bibitem[AM15]{araujo_massarenti}
C.~Araujo and A.~Massarenti, \emph{Explicit log {F}ano structures on blow-ups
  of projective spaces}, preprint arXiv:1505.02460, 2015. To appear in Proc.\ London Math.\ Soc.

\bibitem[Bau91]{bauer}
S.~Bauer, \emph{Parabolic bundles, elliptic surfaces and
  $\text{SU}(2)$-representation spaces of genus zero {F}uchsian groups}, Math.\
  Ann.\ \textbf{290} (1991), 509--526.

\bibitem[BDP16]{BDP}
M.~C. Brambilla, O.~Dumitrescu, and E.~Postinghel, \emph{On the effective cone
  of {$\pr^n$} blown-up at {$n+3$} points}, Exp.\ Math.\ \textbf{25} (2016),
  452--465.

\bibitem[BHK10]{BHK}
I.~Biswas, Y.~I. Holla, and C.~Kumar, \emph{On moduli spaces of parabolic
  vector bundles of rank $2$ on $\mathbb{C}\mathbb{P}^1$}, Michigan Math.\ J.\
  \textbf{59} (2010), 467--479.

\bibitem[Bor90]{borcea90}
C.~Borcea, \emph{Deforming varieties of {$k$}-planes of projective complete
  intersections}, Pacific J.\ Math.\ \textbf{143} (1990), 25--36.

\bibitem[Bor91]{borcea91}
\bysame, \emph{Homogeneous vector bundles and families of {C}alabi-{Y}au
  threefolds. {II}}, Several Complex Variables and Complex Geometry, Part 2
  (Santa Cruz, CA, 1989), Proc.~Symp.~Pure Math., vol.~52, 1991, pp.~83--91.

\bibitem[Cas15]{parabolic}
C.~Casagrande, \emph{\noop{aaa}{R}ank {$2$} quasiparabolic vector bundles on
  {$\pr^1$} and the variety of linear subspaces contained in two
  odd-dimensional quadrics}, Math.\ Z.\ \textbf{280} (2015), 981--988.

\bibitem[CT06]{CoxCT}
A.-M. Castravet and J.~Tevelev, \emph{Hilbert's 14th problem and {C}ox rings},
  Compos.\ Math.\ \textbf{142} (2006), 1479--1498.

\bibitem[DD15]{dolgachevduncan}
I.~V. Dolgachev and A.~Duncan, \emph{Regular pairs of quadratic forms on
  odd-dimensional spaces in characteristic 2}, preprint arXiv:1510.06803, 2015.

\bibitem[DO88]{dolgort}
I.~V. Dolgachev and D.~Ortland, \emph{Point sets in projective spaces and theta
  functions}, Ast{\'e}risque, vol. 165, Soc.\ Math.\ France, 1988.

\bibitem[Dol83]{dolgachev}
I.~V. Dolgachev, \emph{Weyl groups and {C}remona transformations},
  Singularities, Part 1 ({A}rcata, {C}alif., 1981, Proc.~Symp.~Pure Math.,
  vol.~40, 1983, pp.~283--294.

\bibitem[Dol04]{dolgannali}
\bysame, \emph{On certain families of elliptic curves in projective space},
  Ann.\ Mat.\ Pura Appl.\ \textbf{183} (2004), 317--331.

\bibitem[Dol12]{dolgachevbook}
\bysame, \emph{Classical algebraic geometry - a modern view}, Cambridge
  University Press, 2012.

\bibitem[Gre09]{halfcube}
R.~M. Green, \emph{Homology representations arising from the half cube}, Adv.\
  Math.\ \textbf{222} (2009), 216--239.

\bibitem[Gre13]{green}
\bysame, \emph{Combinatorics of minuscule representations}, Cambridge Tracts in
  Mathematics, vol. 199, Cambridge University Press, 2013.

\bibitem[Har92]{harris}
J.~Harris, \emph{Algebraic geometry - a first course}, Graduate Texts in
  Mathematics, vol. 133, Springer-Verlag, 1992.

\bibitem[HK00]{hukeel}
Y.~Hu and S.~Keel, \emph{Mori dream spaces and {GIT}}, Michigan Math.\ J.\
  \textbf{48} (2000), 331--348.

\bibitem[Hum72]{humphreys}
J.~E. Humphreys, \emph{Introduction to {L}ie algebras and representation
  theory}, Graduate Texts in Mathematics, vol.~9, Springer, 1972.

\bibitem[Jia12]{jiang}
Z.~Jiang, \emph{A {N}oether-{L}efschetz theorem for varieties of $r$-planes in
  complete intersections}, Nagoya Math.\ J.\ \textbf{206} (2012), 39--66.

\bibitem[Muk01]{mukaiXIV}
S.~Mukai, \emph{Counterexample to {H}ilbert's fourteenth problem for the
  $3$-dimensional additive group}, RIMS Preprint n.~1343, Kyoto, 2001.

\bibitem[Muk03]{mukaibook}
\bysame, \emph{An introduction to invariants and moduli}, Cambridge Studies in
  Advances Mathematics, vol.~81, Cambridge University Press, 2003.

\bibitem[Muk05]{mukaiADE}
\bysame, \emph{Finite generation of the {N}agata invariant rings in
  {$A$-$D$-$E$} cases}, RIMS Preprint n.~1502, Kyoto, 2005.

\bibitem[Oka16]{okawa_MCD}
S.~Okawa, \emph{On images of {M}ori dream spaces}, Math.\ Ann.\ \textbf{364}
  (2016), 1315--1342.

\bibitem[Rei72]{reidthesis}
M.~Reid, \emph{The complete intersection of two or more quadrics}, Ph.D.
  thesis, University of Cambridge, 1972, available at the author's webpage
  homepages.warwick.ac.uk/~masda/3folds/qu.pdf.

\bibitem[SV05]{soulevoisin}
C.~Soul{\'e} and C.~Voisin, \emph{Torsion cohomology classes and algebraic
  cycles on complex projective manifolds}, Adv.\ Math.\ \textbf{198} (2005),
  107--127.

\end{thebibliography}
\end{document}